\documentclass[a4paper]{article}

\usepackage{amssymb,amsmath,amsfonts}
\usepackage{hyperref}
\usepackage{graphicx,color}
\usepackage{calrsfs}		
\usepackage{amsthm}	
\usepackage[colon,numbers]{natbib}

\usepackage{comment}

\usepackage{geometry}

\usepackage{enumitem}

\newcommand{\lc}{[\![}
\newcommand{\rc}{]\!]}
\newcommand{\1}[1]{{\boldsymbol 1_{\{#1\}}}}
\newcommand{\oo}{\boldsymbol 1}

\newcommand{\E}{\mathbb{E}}
\newcommand{\F}{\mathbb{F}}

\newcommand{\N}{\mathbb{N}}
\renewcommand{\P}{\mathbb{P}}
\newcommand{\Q}{\mathbb{Q}}
\newcommand{\R}{\mathbb{R}}
\newcommand{\dd}{\mathrm{d}}

\newcommand{\Ecal}{{\mathcal E}}
\newcommand{\Fcal}{{\mathcal F}}
\newcommand{\Gcal}{{\mathcal G}}

\newcommand{\Tcal}{{\mathcal T}}

\theoremstyle{plain}
\newtheorem{theorem}{Theorem}
\newtheorem{corollary}[theorem]{Corollary}
\newtheorem{assumption}[theorem]{Assumption}
\newtheorem{lemma}[theorem]{Lemma}

\theoremstyle{definition}
\newtheorem{definition}[theorem]{Definition}

\theoremstyle{definition}
\newtheorem{remark}[theorem]{Remark}
\newtheorem{important remark}[theorem]{Important remark}
\newtheorem{example}[theorem]{Example}

\numberwithin{equation}{section}
\numberwithin{theorem}{section}

\allowdisplaybreaks

\begin{document}

\title{Convergence of Local Supermartingales\thanks{This paper was part of a preprint titled ``Convergence of local supermartingales and Novikov-Kazamaki type conditions for processes with jumps.'' 
We thank Tilmann Bl\"ummel, Pavel Chigansky, Sam Cohen, Christoph Czichowsky, Freddy Delbaen, Moritz D\"umbgen, Hardy Hulley, Jan Kallsen, Ioannis Karatzas, Kostas Kardaras,  Kasper Larsen,  and Nicolas Perkowski for discussions on the subject matter of this paper.  
We are also very grateful to two anonymous referees for their careful reading and helpful comments. 
}}
\author{
Martin Larsson\thanks{Department of Mathematical Sciences, Carnegie Mellon University, Pittsburgh, Pennsylvania 15213, USA. E-mail: martinl@andrew.cmu.edu}
\and
Johannes Ruf\thanks{Department of Mathematics, London School of Economics and Political Science, Columbia House, Houghton St, London WC2A 2AE, United Kingdom. E-mail:     j.ruf@lse.ac.uk}
}

\date{\today}

\maketitle

\begin{abstract}
We characterize the event of convergence of a local supermartingale. Conditions are given in terms of its predictable characteristics and quadratic variation. The notion of stationarily local integrability plays a key role.

\medskip
{\it R\'esum\'e:} Nous caract\'erisons l'\'ev\'enement de convergence d'une surmartingale locale. Les conditions sont exprim\'ees en termes de ses caract\'eristiques pr\'evisibles et de sa variation quadratique. La notion d'int\'egrabilit\'e stationnairement locale joue un r\^ole cl\'e.

\medskip
{\bf Keywords:} Supermartingale convergence, stationary localization.

{\bf MSC2010 subject classification:} Primary 60G07; secondary: 60G17, 60G44. 
\end{abstract}

\section{Introduction}

Among the most fundamental results in the theory of martingales are the martingale and supermartingale convergence theorems of \citet{Doob:1953}. One of Doob's results states that if $X$ is a nonnegative supermartingale, then $\lim_{t\to\infty}X_t$ exists almost surely. If $X$ is not nonnegative, or more generally fails to satisfy suitable integrability conditions, then the limit need not exist, or may only exist with some probability. One is therefore naturally led to search for convenient characterizations of the event of convergence $D=\{\lim_{t\to\infty}X_t \textnormal{ exists in }\R\}$. An archetypical example of such a characterization arises from the Dambis--Dubins--Schwarz theorem: if~$X$ is a continuous local martingale, then $D=\{[X,X]_{\infty-}<\infty\}$ almost surely. This equality fails in general, however, if $X$ is not continuous, in which case it is natural to ask for a description of how the two events differ. The first main goal of the present paper is to address questions of this type: how can one describe the event of convergence of a process $X$, as well as of various related processes of interest? We do this in the setting where $X$ is a {\em local supermartingale on a stochastic interval $\lc0,\tau\lc$}, where $\tau$ is a foretellable time. (Precise definitions are given below, but we remark already here that every predictable time is foretellable.)

While the continuous case is relatively simple, the general case offers a much wider range of phenomena. For instance, there exist locally bounded martingales $X$ for which both $\lim_{t\to \infty}X_t$ exists in $\R$ and $[X,X]_{\infty-}=\infty$, or for which $\liminf_{t\to\infty}X_t=-\infty$, $\limsup_{t\to\infty}X_t=\infty$, and $[X,X]_{\infty-}<\infty$ hold simultaneously almost surely. We provide a large number of examples of this type. To tame this disparate behavior, some form of restriction on the jump sizes is needed. The correct additional property is that of {\em stationarily local integrability}, which is a modification of the usual notion of local integrability.

Our original motivation for considering questions of convergence came from the study of Novikov--Kazamaki type conditions for a nonnegative local martingale $Z=\Ecal(M)$ to be a uniformly integrable martingale. Here $\Ecal(\cdot)$ denotes the stochastic exponential and $M$ is a local martingale. This problem was originally posed by \citet{Girsanov_1960}, and is of great importance in a variety of applications, for example in mathematical finance. An important milestone is due to \citet{Novikov} who proved that if $M$ is continuous, then $\E[e^{\frac{1}{2}[M,M]_{\infty-}}]<\infty$ implies that $Z$ is a uniformly integrable martingale. 

Let us indicate how questions of convergence arise naturally in this context, assuming for simplicity that $M$ is continuous and $Z$ strictly positive, which is the situation studied by~\citet{Ruf_Novikov}. For any bounded stopping time $\sigma$ we have
\[
\E_\P\left[ e^{ \frac{1}{2} [M,M]_\sigma } \right] = \E_\P\left[Z_\sigma e^{ -M_\sigma + [M,M]_\sigma} \right].
\]
While {\em a priori} $Z$ need not be a uniformly integrable martingale, one can still find a probability measure~$\Q$, sometimes called the {\em F\"ollmer measure}, under which $Z$ may explode, say at time $\tau_\infty$, and such that $\dd \Q / \dd \P |_{\mathcal F_\sigma} = Z_\sigma$ holds for any bounded stopping time $\sigma<\tau_\infty$, see \citet{Perkowski_Ruf_2014}. For such stopping times,
\[
\E_\P\left[ e^{ \frac{1}{2} [M,M]_\sigma } \right] = \E_\Q\left[ e^{X_\sigma} \right],
\]
where $X=-M+[M,M]$ is a local $\Q$--martingale on $\lc0,\tau_\infty\lc$. The key point is that  $Z$ is a uniformly integrable martingale under $\P$ if and only if $\Q(\lim_{t\to\tau_\infty}X_t\textnormal{ exists in }\R)=1$. The role of Novikov's condition is to guarantee that the latter holds. In the continuous case there is not much more to say; it is the extension of this methodology to the general jump case that requires more sophisticated convergence criteria for the process $X$, as well as for certain related processes. Moreover, the fact that $\tau_\infty$ may {\em a priori} be finite explains why we explicitly allow $X$ to be defined on a stochastic interval when we work out the theory. We develop this approach in the companion paper \citet{Larsson:Ruf:NK}, where we formulate a necessary and sufficient Novikov-Kazamaki-type condition.

Besides the literature mentioned in the first paragraph of the introduction, \citet{Chow:1963, Chow:1965}, \citet{Robbins:Siegmund:1971}, \citet{Rao:1979}, and \citet{Kruglov:2008} provide related results. 
In this paper, we focus on almost sure convergence. We do not discuss convergence in probability or distribution, but refer the interested reader to \citet{Baez-Duarte}, \citet{Gilat:1972}, and \citet{Pitman:2015}.

The rest of the paper is organized as follows. Section~\ref{S:prelim} contains notational conventions and mathematical preliminaries. Section~\ref{S:EL} introduces the notion of stationary localization and establishes some general properties. Our main convergence theorems and a number of corollaries are given in Section~\ref{S:convergence}. Section~\ref{S:examp} contains counterexamples illustrating the sharpness of the results obtained in Section~\ref{S:convergence}.

\section{Notation and preliminaries} \label{S:prelim}

In this section we establish some basic notation that will be used throughout the paper. For further details and definitions the reader is referred to \citet{JacodS}.

We work on a stochastic basis $(\Omega,\mathcal F, \mathbb F, \P)$ where $\mathbb F=(\Fcal_t)_{t\ge0}$ is a right-continuous filtration, not necessarily augmented by the $\P$--nullsets.  Given a c\`adl\`ag process $X=(X_t)_{t\ge0}$ we write $X_-$ for its left limits and $\Delta X=X-X_-$ for its jump process, using the convention $X_{0-}=X_0$.  The jump measure of $X$ is denoted by $\mu^X$, and its compensator by $\nu^X$. We let $X^\tau$ denote the process stopped at a stopping time~$\tau$. If $X$ is a semimartingale, $X^c$ denotes its continuous local martingale part, and $H\cdot X$ is the stochastic integral of an $X$--integrable process $H$ with respect to $X$. The stochastic integral of a predictable function $F$ with respect to a random measure $\mu$ is denoted $F*\mu$. For two stopping times $\sigma$ and $\tau$, the stochastic interval $\lc\sigma,\tau\lc$ is the set
\[
\lc\sigma,\tau\lc = \{ (\omega,t)\in\Omega\times\R_+ : \sigma(\omega)\le t<\tau(\omega)\}.
\]
Stochastic intervals such as $\rc\sigma,\tau\rc$ are defined analogously. Note that all stochastic intervals are disjoint from $\Omega\times\{\infty\}$.

A process on a stochastic interval $\lc0,\tau\lc$, where $\tau$ is a stopping time, is a measurable map $X:\lc0,\tau\lc\to\overline\R$. We also view $X$ as a process on $\lc0,\infty\lc$ by setting $X_t=0$ for all $t\ge\tau$. 
In this paper, $\tau$ will be a foretellable time; that is, a $[0,\infty]$--valued stopping time that admits a nondecreasing sequence $(\tau_n)_{n \in \N}$ of stopping times with $\tau_n<\tau$ almost surely for all $n \in \N$ on the event $\{\tau>0\}$ and $\lim_{n \to \infty} \tau_n = \tau$ almost surely. Such a sequence is called an announcing sequence. 

If $\tau$ is a foretellable time and $X$ is a process on $\lc0,\tau\lc$, we say that $X$ is a semimartingale on $\lc0,\tau\lc$ if there exists an announcing sequence $(\tau_n)_{n\in \N}$ for $\tau$ such that $X^{\tau_n}$ is a semimartingale for each $n \in \N$. Local martingales and local supermartingales on $\lc0,\tau\lc$ are defined analogously. Basic notions for semimartingales carry over by localization to semimartingales on stochastic intervals. For instance, if $X$ is a semimartingale on $\lc0,\tau\lc$, its quadratic variation process $[X,X]$ is defined as the process on $\lc0,\tau\lc$ that satisfies $[X,X]^{\tau_n} = [X^{\tau_n}, X^{\tau_n}]$ for each $n \in \N$. Its jump measure~$\mu^X$ and compensator~$\nu^X$ are defined analogously, as are stochastic integrals with respect to $X$ (or $\mu^X$, $\nu^X$, $\mu^X-\nu^X$). In particular, $H$ is called $X$--integrable if it is $X^{\tau_n}$--integrable for each $n\in\N$, and $H\cdot X$ is defined as the semimartingale on $\lc 0, \tau\lc$ that satisfies $(H \cdot X)^{\tau_n} = H \cdot X^{\tau_n}$ for each $n \in \N$. Similarly, $G_{\rm loc}(\mu^X)$ denotes the set of predictable functions $F$ for which the compensated integral $F*(\mu^{X^{\tau_n}}-\nu^{X^{\tau_n}})$ is defined for each $n\in\N$ (see Definition~II.1.27 in~\citet{JacodS}), and $F*(\mu^X-\nu^X)$ is the semimartingale on $\lc0,\tau\lc$ that satisfies $(F*(\mu^X-\nu^X))^{\tau_n}=F*(\mu^{X^{\tau_n}}-\nu^{X^{\tau_n}})$ for all $n\in\N$. One easily verifies that all these notions are independent of the particular sequence $(\tau_n)_{n\in\N}$. We refer to \citet{Maisonneuve1977}, \citet{Jacod_book}, and Appendix~A in~\citet{CFR2011} for further details on local martingales on stochastic intervals.

Since we do not require $\Fcal$ to contain all $\P$--nullsets, we may run into measurability problems with quantities like $\sup_{t<\tau}X_t$ for an optional (predictable, progressive) process $X$ on $\lc0,\tau\lc$. However, the left-continuous process $\sup_{t<\cdot}X_t$ is adapted to the $\P$--augmentation $\overline\F$ of $\F$; see the proof of Theorem~IV.33 in~\citet{Dellacherie/Meyer:1978}. Hence it is $\overline \F$--predictable, so we can find an $\F$--predictable process $U$ that is indistinguishable from it; see Lemma~7 in Appendix~1 of~\citet{Dellacherie/Meyer:1982}. Thus the process $V=U\vee X$ is $\F$-optional (predictable, progressive) and indistinguishable from $\sup_{t\le\cdot}X_t$. When writing the latter, we always refer to the indistinguishable process $V$.

We define the set
\[
\Tcal = \{ \tau : \textnormal{ $\tau$ is a bounded stopping time} \}.
\]
Finally, we emphasize the convention $Y(\omega)\oo_A(\omega)=0$ for all (possibly infinite-valued) random variables~$Y$, events $A\in\Fcal$, and $\omega\in\Omega\setminus A$.

\section{The notion of stationary localization} \label{S:EL}
The following strengthening of the notion of local integrability and boundedness turns out to be very useful. It is a mild variation of the notion of $\gamma$-localization  by~\citet{CS:2005}.

\begin{definition}[Stationarily locally integrable / bounded]  \label{D:stationarily}
Let $\tau$ be a foretellable time and $X$ a progressive process on $\lc0,\tau\lc$. Let $D\in\Fcal$. We call $X$ {\em stationarily locally integrable on $D$} if there exists a nondecreasing sequence $(\rho_n)_{n\in\N}$ of stopping times as well as a sequence $(\Theta_n)_{n \in \N}$ of integrable random variables such that the following two conditions hold almost surely:\begin{enumerate}
\item $\sup_{t\ge0} |X^{\rho_n}_t|\le\Theta_n$ for each $n\in\N$.
\item\label{D:stationarily:ii} $D\subset\bigcup_{n\in\N}\{\rho_n \geq \tau\}$.
\end{enumerate}
If $D=\Omega$, we simply say that $X$ is {\em stationarily locally integrable}. Similarly, we call $X$ {\em stationarily locally bounded (on $D$)} if $\Theta_n$ can be taken deterministic for each $n \in \N$. \qed
\end{definition}

Stationary localization naturally suggests itself when one deals with questions of convergence. The reason is the simple inclusion $D\subset\bigcup_{n\in\N}\{X_t=X^{\rho_n}_t \textnormal{ for all } t \geq 0\}$, where $D$ and $(\rho_n)_{n \in \N}$ are as in Definition~\ref{D:stationarily}. This inclusion shows that to prove that $X$ converges on $D$, it suffices to prove that each $X^{\rho_n}$ converges on $D$. If $X$ is stationarily locally integrable on $D$, one may thus assume when proving such results that $X$ is in fact uniformly bounded by an integrable random variable. This stationary localization procedure will be used repeatedly throughout the paper.

It is clear that a process is stationarily locally integrable if it is stationarily locally bounded.  We now provide some further observations on this strengthened notion of localization.

\begin{lemma}[Properties of stationary localization]  \label{L:ELI}
Let $\tau$ be a foretellable time, $D\in\Fcal$, and $X$ a process on $\lc0,\tau\lc$.
\begin{enumerate}
\item\label{L:ELI:1} If $X=X'+X''$, where $X'$ and $X''$ are stationarily locally integrable (bounded) on $D$, then $X$ is stationarily locally integrable (bounded) on $D$.
\item\label{L:ELI:2} If there exists a nondecreasing sequence $(\rho_n)_{n \in \N}$ of stopping times with $D\subset\bigcup_{n\in\N}\{\rho_n \geq \tau\}$ such that $X^{\rho_n}$ is stationarily locally integrable (bounded) on $D$ for each $n \in \N$, then $X$ is stationarily locally integrable (bounded) on $D$. 
\item\label{L:ELI:3} Suppose $X$ is c\`adl\`ag adapted. Then $\sup_{t<\tau}|X_t|<\infty$ on $D$ and $\Delta X$ is stationarily locally integrable (bounded) on $D$ if and only if $X$ is stationarily locally integrable (bounded) on $D$.
\item\label{L:ELI:5} Suppose $X$ is c\`adl\`ag adapted. Then $x \oo_{x > 1} * \mu^X$ is stationarily locally integrable on $D$ if and only if $x \oo_{x > 1} * \nu^X$ is stationarily locally integrable on $D$.  Any of these two conditions imply that $(\Delta X)^+$ is stationarily locally integrable on $D$.
\item\label{L:ELI:6} Suppose $X$ is optional. If $\sup_{\sigma\in\Tcal}\E[|X_\sigma|\1{\sigma<\tau}]<\infty$ then $X$ is stationarily locally integrable.
\item\label{L:ELI:4} Suppose $X$ is predictable. Then  $\sup_{t<\tau}|X_t|<\infty$ on $D$ if and only if $X$ is stationarily locally bounded on $D$ if and only if $X$ is stationarily locally integrable on $D$.
\end{enumerate}
\end{lemma}

\begin{proof}
The statement in \ref{L:ELI:1} follows by defining a sequence $(\rho_n)_{n \in \N}$ of stopping times by $\rho_n = \rho_n' \wedge \rho_n''$, where $(\rho_n')_{n \in \N}$ and $(\rho_n'')_{n \in \N}$ localize $X'$ and $X''$ stationarily.  For \ref{L:ELI:2}, suppose without loss of generality that $\rho_n\le\tau$ for all $n\in\N$, and let $(\rho_m^{(n)})_{m \in \N}$ localize $X^{\rho_n}$ stationarily, for each $n \in \N$.  Let $m_n$ be the smallest index such that $\P(D \cap \{\rho_{m_n}^{(n)} < \rho_n\}) \leq 2^{-n}$ for each $n \in \N$. Next, define  $\widehat\rho_0=0$ and then iteratively $\widehat\rho_n = \rho_n \wedge (\rho_{m_n}^{(n)} \vee \widehat\rho_{n-1})$ for each $n \in \N$, and check, by applying Borel-Cantelli, that the sequence $(\widehat \rho_n)_{n \in \N}$ satisfies the conditions of Definition~\ref{D:stationarily}.

For \ref{L:ELI:3} define the sequence $(\rho_n)_{n \in \N}$ of crossing times by $\rho_n=\inf\{t \geq 0:|X_t|\ge n\}$. 
Note also the inequalities $|X^{\rho_n}|\le n+|\Delta X_{\rho_n}|\1{\rho_n<\tau}$ and $|\Delta X^{\rho_n}|\le 2 n+|X^{\rho_n}|$ for each $n \in \N$.  This yields the equivalence for  $X^{\rho_n}$ for each $n \in \N$  and the statement follows by applying \ref{L:ELI:2}.

To prove~\ref{L:ELI:5}, suppose first $x\oo_{x>1}*\mu^X$ is stationarily locally integrable on~$D$. In view of~\ref{L:ELI:2} we may assume by localization that it is dominated by some integrable random variable~$\Theta$, which then yields $\E[x\oo_{x>1}*\nu^X_{\tau-}]\le\E[\Theta]<\infty$. Thus $x\oo_{x>1}*\nu^X$ is dominated by the integrable random variable $x\oo_{x>1}*\nu^X_{\tau-}$, as required. For the converse direction simply interchange $\mu^X$ and $\nu^X$. The fact that $(\Delta X)^+ \leq 1 + x \oo_{x > 1} * \mu^X$ then allows us to conclude.

We now prove \ref{L:ELI:6}, supposing without loss of generality that $X\ge0$. Let $\overline\Fcal$ be the $\P$-completion of $\Fcal$, and write $\P$ also for its extension to $\overline\Fcal$. Define $C=\{\sup_{t<\tau}X_t=\infty\}\in\overline\Fcal$. We first show that $\P(C)=0$, and assume for contradiction that $\P(C)>0$. For each $n\in\N$ define the optional set $O_n=\{t<\tau \textnormal{ and } X_t\ge n\}\subset\Omega\times\R_+$. Then $C=\bigcap_{n\in\N}\pi(O_n)$, where $\pi(O_n)\in\overline \Fcal$ is the projection of $O_n$ onto $\Omega$. The optional section theorem, see Theorem~IV.84 in~\citet{Dellacherie/Meyer:1978}, implies that for each $n\in\N$ there exists a stopping time $\sigma_n$ such that
\begin{equation} \label{eq:L:ELI:section}
\lc\sigma_n\rc\subset O_n \qquad\textnormal{and}\qquad \P\left( \{\sigma_n=\infty\}\cap\pi(O_n) \right) \le \frac{1}{2}\,\P(C).
\end{equation}
Note that the first condition means that $\sigma_n<\tau$ and $X_{\sigma_n}\ge n$ on $\{\sigma_n<\infty\}$ for each $n \in \N$. Thus,
\[
\E[X_{m\wedge\sigma_n}\1{m\wedge\sigma_n<\tau}]
\ge\E[X_{\sigma_n}\oo_{\{\sigma_n\le m\}\cap C}]
\ge n\P(\{\sigma_n<m\}\cap C) \to n\P(\{\sigma_n<\infty\}\cap C)
\]
as $m\to\infty$ for each $n \in \N$. By hypothesis, the left-hand side is bounded by a constant $\kappa$ that does not depend on~$m \in \N$ or~$n \in \N$. Hence, using  that $C\subset\pi(O_n)$ for each $n \in \N$ as well as \eqref{eq:L:ELI:section}, we get
\[
\kappa \ge n\P(\{\sigma_n<\infty\}\cap C) \ge n\Big( \P(C) - \P(\{\sigma_n=\infty\}\cap \pi(O_n))\Big) \ge \frac{n}{2}\,\P(C).
\]
Letting $n$ tend to infinity, this yields a contradiction, proving $\P(C)=0$ as desired. Now define $\rho_n=\inf\{t \geq 0 :X_t\ge n\} \wedge n$ for each $n\in\N$. By what we just proved, $\P(\bigcup_{n\in\N}\{\rho_n\ge\tau\})=1$. Furthermore, for each $n\in\N$ we have $0\le X^{\rho_n}\le n+X_{\rho_n}\1{\rho_n<\tau}$, which is integrable by assumption and an application of Fatou's lemma. Thus $X$ is stationarily locally integrable.

For \ref{L:ELI:4}, let $U=\sup_{t<\cdot}|X_t|$. It is clear that stationarily local boundedness on $D$ implies stationarily local integrability on $D$ implies $U_{\tau-}<\infty$ on $D$. Hence it suffices to prove that $U_{\tau-}<\infty$ on $D$ implies stationarily local boundedness on $D$. To this end, we may assume that $\tau < \infty$, possibly after a change of time. We now define a process $U'$ on $\lc 0, \infty \lc$ by $U' = U \oo_{\lc 0, \tau\lc} + U_{\tau-} \oo_{\lc  \tau, \infty\lc} $, and follow the proof of Lemma~I.3.10 in \citet{JacodS} to conclude.
\end{proof}

An anonymous referee pointed out that the implication in Lemma~\ref{L:ELI}\ref{L:ELI:6} fails if $X$ is not optional, but only progressive. Indeed, IV.91 in~\citet{Dellacherie/Meyer:1978} contains an example of a progressive set $H$ with almost surely uncountable sections, and still containing no graph of a stopping time. Any process of the form $X=Y\oo_H$ for some progressive process $Y$ then satisfies the hypothesis of Lemma~\ref{L:ELI}\ref{L:ELI:6}, but can easily be constructed to fail stationarily local integrability.


\begin{example}
	If $X$ is a uniformly integrable martingale then $X$ is stationarily locally integrable.  This can be seen by considering first crossing times of $|X|$, as in the proof of Lemma~\ref{L:ELI}\ref{L:ELI:3}. \qed 
\end{example}

\section{Convergence of local supermartingales} \label{S:convergence}

In this section we state and prove a number of theorems regarding the event of convergence of a local supermartingale on a stochastic interval. The results are stated in Subsections~\ref{S:conv statements} and~\ref{S:conv locmg}, while the remaining subsections contain the proofs.

\subsection{Convergence results in the general case} \label{S:conv statements}

Our general convergence results will be obtained under the following basic assumption.

\begin{assumption} \label{A:1}
It is assumed that $\tau>0$ be a foretellable time with announcing sequence  $(\tau_n)_{n \in \N}$ and $X = M-A$ a local supermartingale on~$\lc 0,\tau\lc$, where $M$ and $A$ are a local martingale and a nondecreasing predictable process on $\lc 0,\tau\lc$, respectively, both starting at zero. 
\end{assumption}

\begin{theorem}[Characterization of the event of convergence] \label{T:conv}
Suppose Assumption~\ref{A:1} holds and fix $D\in\Fcal$. The following conditions are equivalent:
\begin{enumerate}[label={\rm(\alph{*})}, ref={\rm(\alph{*})}]
\item\label{T:conv:a} $\lim_{t\to\tau}X_t$ exists in $\R$ on $D$ and $(\Delta X)^- \wedge X^-$ is stationarily locally integrable on $D$.
\item\label{T:conv:a'} $\liminf_{t\to\tau}X_t>-\infty$  on $D$ and $(\Delta X)^- \wedge X^-$ is stationarily locally integrable on $D$.
\item\label{T:conv:b'} $X^-$  is stationarily locally integrable on $D$.
\item\label{T:conv:b''} $X^+$  is stationarily locally integrable on $D$ and $A_{\tau-} < \infty$ on $D$.
\item\label{T:conv:e} $X$ is stationarily locally integrable on $D$.
\item\label{T:conv:c} $[X^c,X^c]_{\tau-} + (x^2\wedge|x|)*\nu^X_{\tau-} + A_{\tau-} <\infty$ on $D$.
\item\label{T:conv:f} $[X,X]_{\tau-}< \infty$ on $D$, $\limsup_{t\to\tau}X_t>-\infty$  on $D$, and $(\Delta X)^- \wedge X^-$ is stationarily locally integrable on $D$.
\end{enumerate}
If additionally  $X$ is constant after $\tau_J=\inf\{t \geq 0:\Delta X_t=-1\}$, the above conditions are equivalent to the following condition:
	\begin{enumerate}[resume, label={\rm(\alph{*})}, ref={\rm(\alph{*})}]
		\item\label{T:conv:g} Either $\lim_{t\to\tau}\Ecal(X)_t \textnormal{ exists in } \R\setminus\{0\}$ or  $\tau_J < \tau$ on $D$, and $(\Delta X)^- \wedge X^-$ is stationarily locally integrable on $D$.
	\end{enumerate}
\end{theorem}

\begin{remark} \label{R:4.3}
We make the following observations concerning Theorem~\ref{T:conv}. As in the theorem, we suppose Assumption~\ref{A:1} holds and fix $D\in\Fcal$:
\begin{itemize}
\item For any local supermartingale $X$, the jump process $\Delta X$ is locally integrable. This is however not enough to obtain good convergence theorems as the examples in Section~\ref{S:examp} show. The crucial additional assumption is that localization be in the stationary sense. In Subsections~\ref{A:SS:lack} and \ref{A:SS:one}, several examples are collected that illustrate that the conditions of Theorem~\ref{T:conv} are non-redundant, in the sense that the implications fail for some local supermartingale $X$ if some of the conditions is omitted.
\item If any of the conditions \ref{T:conv:a}--\ref{T:conv:f} holds then $\Delta X$ is stationarily locally integrable on $D$. This is a by-product of the proof of the theorem. The stationarily local integrability of $\Delta X$ also follows, {\em a posteriori}, from Lemma~\ref{L:ELI}\ref{L:ELI:3}.
\item If any of the conditions \ref{T:conv:a}--\ref{T:conv:f} holds and if  $X = M^\prime - A^\prime$ for some local supermartingale $M^\prime$ and some nondecreasing (not necessarily predictable) process $A'$ with $A'_0 = 0$, then $\lim_{t\to\tau} M^\prime_t$ exists in~$\R$ on $D$ and $A^\prime_{\tau-} < \infty$ on $D$. Indeed, $M^\prime \geq X$ and thus the implication {\ref{T:conv:b'}} $\Longrightarrow$  {\ref{T:conv:a}} applied to $M^\prime$ yields that $\lim_{t \to \tau} M^\prime_t$ exists in $\R$, and therefore also $A_{\tau-}^\prime<\infty$.
\item One might conjecture that Theorem~\ref{T:conv} can be generalized to special semimartingales $X=M+A$ on~$\lc0,\tau\lc$ by replacing $A$ with its total variation process ${\rm Var}(A)$ in {\ref{T:conv:b''}} and {\ref{T:conv:c}}. However, such a generalization is not possible in general. As an illustration of what can go wrong, consider the deterministic finite variation process $X_t=A_t=\sum_{n=1}^{[t]}(-1)^n n^{-1}$, where $[t]$ denotes the largest integer less than or equal to~$t$. Then $\lim_{t\to \infty}X_t$ exists in~$\R$, being an alternating series whose terms converge to zero. Thus {\ref{T:conv:a}}--{\ref{T:conv:b'}} \& \ref{T:conv:e} \& \ref{T:conv:f} hold with $D=\Omega$. However, the total variation ${\rm Var}(A)_{\infty-}=\sum_{n=1}^\infty n^{-1}$ is infinite, so {\ref{T:conv:b''}} \& {\ref{T:conv:c}} do not hold with $A$ replaced by ${\rm Var}(A)$. Related questions are addressed by~\citet{CS:2005}.
\item One may similarly ask about convergence of local martingales of the form $X=x*(\mu-\nu)$ for some integer-valued random measure $\mu$ with compensator $\nu$. Here nothing can be said in general in terms of $\mu$ and $\nu$; for instance, if $\mu$ is already predictable then $X=0$.
\qed
\end{itemize}
\end{remark}

Theorem~\ref{T:conv} is stated in a  general form and its power appears when one considers specific events $D \in \mathcal F$.  For example, we may let $D = \{\lim_{t\to\tau}X_t \textnormal{ exists in }\R\}$ or $D = \{\liminf_{t\to\tau}X_t>-\infty\}$.  Choices of this kind lead directly to the following corollary.

\begin{corollary}[Stationarily local integrability from below] \label{C:conv2}
Suppose Assumption~\ref{A:1} holds and $(\Delta X)^- \wedge X^-$ is stationarily locally integrable on $\{ \textnormal{$\limsup_{t\to\tau}X_t>-\infty$} \}$. Then the following events are almost surely equal:
 \begin{align}
\label{T:conv2:1}
&\Big\{ \textnormal{$\lim_{t\to\tau}X_t$ exists in $\R$} \Big\}; \\
&\Big\{ \textnormal{$\liminf_{t\to\tau}X_t>-\infty$} \Big\};\label{T:conv2:6}\\
\label{T:conv2:2}
&\Big\{ \textnormal{$[X^c,X^c]_{\tau-} + (x^2 \wedge |x|) * \nu_{\tau-}^X + A_{\tau-} <\infty$} \Big\}; \\
&\Big\{ \textnormal{$[X,X]_{\tau-} <\infty$} \Big\} \cap \Big\{\textnormal{$\limsup_{t\to\tau}X_t > -\infty$} \Big\}.\label{T:conv2:5}
\end{align}
\end{corollary}
\begin{proof}
The statement follows directly from Theorem~\ref{T:conv}, where for each inclusion the appropriate event $D$ is fixed.
\end{proof}

We remark that the identity \eqref{T:conv2:1} $=$ \eqref{T:conv2:6} appears already in Theorem~5.19 of \citet{Jacod_book} under slightly more restrictive assumptions, along with the equality
\begin{equation} \label{eq:XMA}
	\Big\{ \textnormal{$\lim_{t\to\tau}X_t$ exists in $\R$} \Big\} = \Big\{ \textnormal{$\lim_{t\to\tau}M_t$ exists in $\R$} \Big\} \cap \Big\{ A_{\tau-} < \infty \Big\}.
\end{equation}
Corollary~\ref{C:conv2} yields that this equality in fact holds under assumptions strictly weaker than in \citet{Jacod_book}. Note, however, that some assumption is needed; see Example~\ref{ex:semimartingale}. Furthermore, a special case of the equivalence  \ref{T:conv:f}  $\Longleftrightarrow$  \ref{T:conv:g}  in Theorem~\ref{T:conv} appears in Proposition~1.5 of \citet{Lepingle_Memin_Sur}. Moreover, under additional integrability assumptions on the jumps, Section~4 in \citet{Kabanov/Liptser/Shiryaev:1979} provides related convergence conditions. In general, however, we could not find any of the implications  in Theorem~\ref{T:conv}---except, of course, the trivial implication \ref{T:conv:a} $\Longrightarrow$ \ref{T:conv:a'}---in this generality in the literature. Some of the implications are easy to prove, some of them are more involved. Some of these implications were expected, while others were surprising to us; for example, the limit superior in \ref{T:conv:f} is needed even if~$A=0$ so that~$X$ is a local martingale on $\lc0,\tau\lc$. Of course, whenever the stationarily local integrability condition appears, then, somewhere in the corresponding proof, so does a reference to the classical supermartingale convergence theorem, which relies on Doob's upcrossing inequality.

\begin{corollary}[Stationarily local integrability]   \label{C:convergence_QV}
Under Assumption~\ref{A:1}, if $|\Delta X| \wedge |X|$ is stationarily locally integrable we have, almost surely,
\begin{align*}
\Big\{ \textnormal{$\lim_{t\to\tau}X_t$ exists in $\R$} \Big\} = \Big\{ \textnormal{$[X,X]_{\tau-} <\infty$} \Big\} \cap \Big\{ A_{\tau-}<\infty \Big\}.
\end{align*}
\end{corollary}

\begin{proof}
 The inclusion ``$\subset$'' is immediate from \eqref{T:conv2:1} $\subset$ \eqref{T:conv2:2} $\cap$ \eqref{T:conv2:5} in Corollary~\ref{C:conv2}. For the reverse inclusion, 
note that $\{[X,X]_{\tau-}<\infty\}=\{[M,M]_{\tau-}<\infty\}$ and $|\Delta M| \wedge |M|$ is stationarily locally integrable on $\{A_{\tau-}<\infty\}$, by Lemma~\ref{L:ELI}\ref{L:ELI:4}. In view of~\eqref{eq:XMA}, it suffices now to show that 
\[
	 \Big\{ \textnormal{$[M,M]_{\tau-} <\infty$} \Big\} \cap  \Big\{ A_{\tau-}<\infty \Big\}\subset \Big\{ \textnormal{$\lim_{t\to\tau}M_t$ exists in $\R$} \Big\}.
\]
To this end, note that
\begin{align*}
\Big\{ \textnormal{$[M,M]_{\tau-} <\infty$} \Big\} &= \left(\Big\{ \textnormal{$[M,M]_{\tau-} <\infty$} \Big\} \cap \Big\{ \limsup_{t\to \tau } M_t > -\infty \Big\}\right)\\
&\qquad \cup \left(\Big\{ \textnormal{$[M,M]_{\tau-} <\infty$} \Big\} \cap \Big\{ \limsup_{t\to \tau } M_t = -\infty \Big\} \cap \Big\{ \limsup_{t\to \tau } (-M_t) > -\infty \Big\}\right).
\end{align*}
We obtain now the desired inclusion by applying the implication  {\ref{T:conv:f}} $\Longrightarrow$  {\ref{T:conv:a}} in Theorem~\ref{T:conv}  once to $M$ and once to $-M$.
\end{proof}

\begin{corollary}[$L^1$--boundedness]    \label{C:conv001}
Suppose Assumption~\ref{A:1} holds, and let $f:\R\to\R_+$ be any nondecreasing function with $f(x)\ge x$ for all sufficiently large $x$. Then the following conditions are equivalent:
\begin{enumerate}[label={\rm(\alph*)},ref={\rm(\alph*)}]
\item\label{T:conv1:a} $\lim_{t\to\tau}X_t$ exists in $\R$  and $(\Delta X)^- \wedge X^-$ is stationarily locally integrable.
\item\label{T:conv1:c} $A_{\tau-}<\infty$ and for some  stationarily locally integrable optional process $U$,
\begin{align}  \label{eq:T:conv:exp}
\sup_{\sigma \in \mathcal{T}} \E\left[f(X_\sigma - U_\sigma)  \oo_{\{\sigma<\tau\}} \right] < \infty.
\end{align}
\item\label{T:conv1:d} For some  stationarily locally integrable optional process $U$, \eqref{eq:T:conv:exp} holds with $x\mapsto f(x)$ replaced by $x\mapsto f(-x)$.
\item\label{C:conv001:e}  The process $\overline X=X\oo_{\lc0,\tau\lc}+(\limsup_{t\to\tau}X_t)\oo_{\lc\tau,\infty\lc}$, extended to $[0,\infty]$ by $\overline X_\infty  = \limsup_{t\to\tau}X_t$,  is a semimartingale on $[0,\infty]$ and $(\Delta X)^- \wedge X^-$ is stationarily locally integrable.
\item\label{C:conv001:d} The process $\overline X=X\oo_{\lc0,\tau\lc}+(\limsup_{t\to\tau}X_t)\oo_{\lc\tau,\infty\lc}$,  extended to $[0,\infty]$ by $\overline X_\infty  = \limsup_{t\to\tau}X_t$,  is a special semimartingale on $[0,\infty]$.
\end{enumerate}
\end{corollary}

\begin{proof}
{\ref{T:conv1:a}} $\Longrightarrow$ {\ref{T:conv1:c}} \& {\ref{T:conv1:d}}: Thanks to the implication {\ref{T:conv:a}} $\Longrightarrow$ {\ref{T:conv:b''}} \& {\ref{T:conv:e}} in Theorem~\ref{T:conv}  we may simply take $U=X$.

{\ref{T:conv1:c}} $\Longrightarrow$ {\ref{T:conv1:a}}:
We have $f(x)\ge \1{x\ge \kappa}x^+$ for some constant $\kappa \geq 0$ and all $x\in\R$. Hence~\eqref{eq:T:conv:exp} holds with $f(x)$ replaced by $x^+$. Lemma~\ref{L:ELI}\ref{L:ELI:6} then implies that $(X-U)^+$ is stationarily locally integrable. Since $X^+\le (X-U)^++U^+$, we have $X^+$ is stationarily locally integrable.
The implication \ref{T:conv:b''} $\Longrightarrow$ \ref{T:conv:a} in Theorem~\ref{T:conv} now yields~\ref{T:conv1:a}.

{\ref{T:conv1:d}} $\Longrightarrow$ {\ref{T:conv1:c}}:
We now have $f(x)\ge \1{x\le -\kappa}x^-$ for some constant $\kappa \geq 0$ and all $x\in\R$, whence as above, $(X-U)^-$ is stationarily locally integrable. Since $M^-\le(M-U)^-+U^-\le(X-U)^-+U^-$, it follows that $M^-$ is stationarily locally integrable. The implication \ref{T:conv:b'} $\Longrightarrow$ \ref{T:conv:a}~\&~\ref{T:conv:e} in Theorem~\ref{T:conv} yields that $\lim_{t\to\tau}M_t$ exists in $\R$ and that $M$ is stationarily locally integrable. Hence $A=(U-X+M-U)^+\le(X-U)^-+|M|+|U|$ is stationarily locally integrable, so Lemma~\ref{L:ELI}\ref{L:ELI:4} yields $A_{\tau-}<\infty$. Thus~{\ref{T:conv1:c}} holds.

\ref{T:conv1:a}   $\Longrightarrow$ \ref{C:conv001:e}:
By \eqref{eq:XMA}, $A$ and $M$ converge. Moreover, since $\Delta M \geq \Delta X$, we have  $(\Delta M)^-$ is stationarily locally integrable by Remark~\ref{R:4.3}, say with localizing sequence $(\rho_n)_{n \in \N}$. Now, it is sufficient to prove that $M^{\rho_n}$ is a local martingale on $[0,\infty]$ for each $n \in \N$, which, however, follows from Lemma~\ref{L:SMC} below.

\ref{C:conv001:e} $\Longrightarrow$ \ref{T:conv1:a}:
Obvious.

\ref{T:conv1:a} \& \ref{C:conv001:e} $\Longleftrightarrow$ \ref{C:conv001:d}:
This equivalence follows from Proposition~II.2.29 in~\citet{JacodS}, in conjunction with the equivalence \ref{T:conv:a}  $\Longleftrightarrow$  \ref{T:conv:c} in Theorem~\ref{T:conv}.
\end{proof}

	Examples~\ref{E:P1} and \ref{ex:semimartingale} below illustrate that the integrability condition is needed in order that \ref{C:conv001}\ref{T:conv1:a} imply the semimartingale property of $X$ on the extended axis. These examples also show that the integrability condition in Corollary~\ref{C:conv001}\ref{C:conv001:e} is not redundant.

\begin{remark} \label{R:bed implied} In Corollary~\ref{C:conv001}, convergence implies not only $L^1$--integrability but also boundedness. Indeed, let $g:\R\to\R_+$ be either $x\mapsto f(x)$ or $x\mapsto f(-x)$. If any of the conditions in Corollary~\ref{C:conv001} holds then there exists a  stationarily locally integrable optional process $U$ such that the family
\begin{align*}
\left(g(X_\sigma - U_\sigma)  \oo_{\{\sigma<\tau\}} \right)_{\sigma \in \mathcal{T}} \qquad \textnormal{is bounded.}
\end{align*}
To see this, note that if {\ref{T:conv1:a}} holds then $X$ is stationarily locally integrable. If $g$ is $x\mapsto f(x)$, let $U=\sup_{t \leq \cdot} X_t$, whereas if $g$ is $x\mapsto f(-x)$, let $U=\inf_{t \leq \cdot} X_t$. In either case, $U$ is  stationarily locally integrable and $(g(X_\sigma-U_\sigma))_{\sigma\in\Tcal}$ is bounded.\qed
\end{remark}

With a suitable choice of $f$ and additional requirements on $U$, condition~\eqref{eq:T:conv:exp} has stronger implications for the tail integrability of the compensator $\nu^X$ than can be deduced, for instance, from Theorem~\ref{T:conv} directly. The following result records the useful case where $f$ is an exponential.

\begin{corollary}[Exponential integrability of~$\nu^X$] \label{C:conv1}
Suppose Assumption~\ref{A:1} holds. If $A_{\tau-}<\infty$ and~\eqref{eq:T:conv:exp} holds with some~$U$ that is stationarily locally bounded and with $f(x)=e^{cx}$ for some $c\ge1$, then
\begin{equation}\label{T:conv:nu}
(e^x-1-x) * \nu^X_{\tau-} < \infty.
\end{equation}
\end{corollary}

\begin{proof}
In view of Lemma~\ref{L:ELI}\ref{L:ELI:4} we may assume by localization that $A=U=0$ and by Jensen's inequality that $c = 1$. Lemma~\ref{L:ELI}\ref{L:ELI:6} then implies that $e^X$ and hence $X^+$ is stationarily locally integrable. Thus by Theorem~\ref{T:conv}, $\inf_{t<\tau} X_t >-\infty$. It\^o's formula yields
\begin{align*}
e^X &= 1 + e^{X_-} \cdot X + \frac{1}{2} e^{X_-}\cdot [X^c,X^c] + \left(e^{X_-}(e^x-1-x)\right)*\mu^X.
\end{align*}
 The second term on the right-hand side is a local martingale on $\lc0,\tau\lc$, so we may find a localizing sequence $(\rho_n)_{n\in\N}$ with $\rho_n<\tau$. Taking expectations and using the defining property of the compensator~$\nu^X$ as well as the associativity of the stochastic integral yield
\[
\E\left[ e^{X_{\rho_n}} \right] = 1+  \E\left[ e^{X_-}\cdot\left(\frac{1}{2} [X^c,X^c] + (e^x-1-x)*\nu^X\right)_{\rho_n} \right]
\]
for each $n \in \N$.
Due to~\eqref{eq:T:conv:exp}, the left-hand side is bounded by a constant that does not depend on~$n \in \N$. We now let~$n$ tend to infinity and recall that $\inf_{t<\tau} X_t >-\infty$ to deduce by the monotone convergence theorem that~\eqref{T:conv:nu} holds.
\end{proof}

\begin{remark}  \label{R:implication}
Stationarily local integrability of $U$ is not enough in Corollary~\ref{C:conv1}.  For example, consider an integrable random variable $\Theta$ with $\E[\Theta]= 0$ and $\E[e^\Theta] = \infty$ and the process $X =  \Theta \oo_{\lc 1, \infty\lc}$ under its natural filtration. Then $X$ is a martingale. Now, with $U = - X$, \eqref{eq:T:conv:exp} holds with $f(x)=e^x$, but $(e^x - 1 - x) * \nu^X_{\infty-} = \infty$. \qed
\end{remark}

\subsection{Convergence results with jumps bounded below} \label{S:conv locmg}

We now specialize to the case where $X$ is a local martingale on a stochastic interval with jumps bounded from below. The aim is to study a related process $Y$, which appears naturally in connection with the nonnegative local martingale $\Ecal(X)$. We comment on this connection below.

\begin{assumption} \label{A:2}
It is assumed that $\tau$ be a foretellable time, and $X$ a local martingale on~$\lc 0,\tau\lc$ with $\Delta X>-1$.  It is moreover assumed that $(x-\log(1+x))*\nu^X$ be finite-valued such that
\begin{equation*}
Y=X^c+\log(1+x)*(\mu^X-\nu^X)
\end{equation*}
is well defined.
\end{assumption}

The significance of the process $Y$ originates with the identity
\begin{align} \label{eq:VV}
\Ecal(X)= e^{Y-V} \quad\textnormal{on $\lc0,\tau\lc$,} \quad \textnormal{where} \quad V=\frac{1}{2}[X^c,X^c]+(x-\log(1+x))*\nu^X.
\end{align}
Thus $Y$ is the local martingale part and $-V$ is the predictable finite variation part of the local supermartingale $\log \Ecal(X)$. The process $V$ is called the {\em exponential compensator} of $Y$, and $Y-V$ is called the {\em logarithmic transform} of $X$. These notions play a central role in~\citet{Kallsen_Shir}.

Observe that the jumps of $Y$ can be expressed as
\begin{equation} \label{eq:DY}
\Delta Y_t = \log(1+\Delta X_t) + \gamma_t,
\end{equation}
where
\begin{align*}
	\gamma_t= - \int \log(1+x) \nu^X(\{t\},\dd x)
\end{align*}
for all $t<\tau$. Jensen's inequality and the fact that $\nu^X(\{t\},\R)\le 1$ imply that $\gamma\ge 0$. If $X$ is quasi-left continuous, then $\gamma\equiv 0$.

In the spirit of our previous results, we now present a theorem that relates convergence of the processes $X$ and $Y$ to the finiteness of various derived quantities.

\begin{theorem}[Joint convergence of a local martingale and its logarithmic transform] \label{T:convYX}
Suppose Assumption~\ref{A:2} holds, and fix $\eta\in(0,1)$ and $\kappa>0$. Then the following events are almost surely equal:
\begin{align}
 	&\Big\{ \lim_{t\to\tau}X_t \textnormal{ exists in }\R\Big\} \cap \Big\{ \lim_{t\to\tau}Y_t \textnormal{ exists in }\R\Big\}   \label{T:convYX:1};\\
	&\Big\{\frac{1}{2} [X^c,X^c]_{\tau-}+ (x - \log(1+x)) * \nu^X_{\tau-} < \infty\Big\} ;     \label{T:convYX:2} \\
 	&\Big\{ \lim_{t\to\tau}X_t \textnormal{ exists in }\R\Big\} \cap \Big\{ -\log(1+x)\oo_{x< -\eta}*\nu^X_{\tau-} < \infty\Big\}  \label{T:convYX:3};\\
	&\Big\{ \lim_{t\to\tau}Y_t \textnormal{ exists in }\R\Big\} \cap   \Big\{x\oo_{x>\kappa} * \nu^X_{\tau-} < \infty\Big\}    .  \label{T:convYX:4}
\end{align}
\end{theorem}

\begin{lemma} \label{L:convYX}
Suppose Assumption~\ref{A:2} holds. For any event $D \in \mathcal F$ with $x\oo_{x>\kappa} * \nu^X_{\tau-} < \infty$ on $D$ for some $\kappa>0$, the following three statements are equivalent:
\begin{enumerate}[label={\rm(\alph{*})}, ref={\rm(\alph{*})}]
\item\label{T:convYX:a} $\lim_{t\to\tau}Y_t$ exists in $\R$ on $D$. 
\item\label{T:convYX:b'} $Y^-$  is stationarily locally integrable on $D$.
\item\label{T:convYX:b''}  $Y^+$  is stationarily locally integrable on $D$.
\end{enumerate}
\end{lemma}

\begin{proof}
The implications follow  from Theorem~\ref{T:conv}.  Only that \ref{T:convYX:a} implies \ref{T:convYX:b'} \& \ref{T:convYX:b''} needs an argument, and it suffices to show that $(\Delta (-Y))^-$ is stationarily locally integrable on $D$. By \eqref{eq:DY} we have $(\Delta(-Y))^-\le(\Delta X)^++\gamma$; Lemma~\ref{L:ELI}\ref{L:ELI:5} implies that $(\Delta X)^+$ is stationarily locally integrable; and Lemma~\ref{L:convYX:1} below and Lemma~\ref{L:ELI}\ref{L:ELI:4} imply that $\gamma$ is stationarily locally bounded on~$D$.
\end{proof}

\begin{corollary}[$L^1$--boundedness] \label{C:convYX2}
Suppose Assumption~\ref{A:2} holds and fix $c\ne0$, $\eta\in(0,1)$, and $\kappa>0$. The following conditions are equivalent:
\begin{enumerate}[label={\rm(\alph*)},ref={\rm(\alph*)}]
\item \label{T:convYX2:a} $\lim_{t\to\tau}X_t$ exists in $\R$ and $-\log(1+x)\oo_{x< -\eta}*\nu^X_{\tau-} < \infty$.
\item\label{T:convYX2:b} $x\oo_{x>\kappa} * \nu^X_{\tau-}< \infty$ and for some stationarily locally integrable optional process $U$ on $\lc0,\tau\lc$ we have
\begin{equation}  \label{eq:T:convYX:eYp}
\sup_{\sigma \in \mathcal{T}} \E\left[e^{cY_\sigma - U_\sigma}  \1{\sigma<\tau} \right] < \infty.
\end{equation}
\end{enumerate}
If $c\ge 1$, these conditions are implied by the following:
\begin{enumerate}[label={\rm(\alph*)},ref={\rm(\alph*)},resume]
\item\label{T:convYX2:c} \eqref{eq:T:convYX:eYp} holds for some stationarily locally bounded optional process $U$ on $\lc0,\tau\lc$.
\end{enumerate}
Finally, the conditions \ref{T:convYX2:a}--\ref{T:convYX2:b} imply that $(e^{cY_\sigma-U_\sigma})_{\sigma\in\Tcal}$ is bounded for some stationarily locally integrable optional process $U$ on $\lc0,\tau\lc$.
\end{corollary}

\begin{proof}
The equivalence of \ref{T:convYX2:a} and \ref{T:convYX2:b} is obtained from \eqref{T:convYX:3} = \eqref{T:convYX:4} in Theorem~\ref{T:convYX}. Indeed, Corollary~\ref{C:conv001} with $X$ replaced by $Y$ and $f(x)=e^{cx}$, together with Lemma~\ref{L:convYX}, yield that \ref{T:convYX2:b} holds if and only if \eqref{T:convYX:4} has full probability. In order to prove that \ref{T:convYX2:c} implies~\ref{T:convYX2:b}  we assume that~\eqref{eq:T:convYX:eYp} holds with $c\geq 1$ and $U$ stationarily locally bounded. Corollary~\ref{C:conv1} yields
\[
\left(1-\frac{1}{\kappa}\log (1+\kappa)\right)\, (e^y-1)\oo_{y>\log (1+\kappa)}*\nu^Y_{\tau-}\le(e^y-1-y)*\nu^Y_{\tau-}<\infty,
\]
so by a localization argument using Lemma~\ref{L:ELI}\ref{L:ELI:4} we may assume that $(e^y-1)\oo_{y>\log (1+\kappa)}*\nu^Y_{\tau-}\le \kappa_1$ for some constant $\kappa_1>0$. Now, \eqref{eq:DY} yields
\[
\Delta X \1{\Delta X>\kappa} = \left(e^{\Delta Y - \gamma}-1\right)\1{e^{\Delta Y}>(1+\kappa)e^\gamma} \le (e^{\Delta Y}-1)\1{\Delta Y>\log (1+\kappa)},
\]
whence $\E[x\oo_{x>\kappa}*\nu^X_{\tau-}]\le \E[(e^y-1)\oo_{y>\log (1+\kappa)}*\nu^Y_{\tau-}]\le\kappa_1$. Thus~\ref{T:convYX2:b} holds. The last statement of the corollary follows as in Remark~\ref{R:bed implied} after recalling Lemma~\ref{L:convYX}.
\end{proof}

\subsection{Some auxiliary results}
In this subsection, we collect some observations that will be useful for the proofs of the convergence theorems of the previous subsection.

\begin{lemma}[Supermartingale convergence] \label{L:SMC}
Under Assumption~\ref{A:1}, suppose $\sup_{n \in \N} \E[X^-_{\tau_n}]<\infty$. Then the limit $G=\lim_{t\to\tau}X_t$ exists in~$\R$ and the process $\overline X=X\oo_{\lc0,\tau\lc} + G\oo_{\lc\tau,\infty\lc}$, extended to $[0,\infty]$ by $\overline X_\infty=G$, is a supermartingale on $[0,\infty]$ and stationarily locally integrable. If, in addition, $X$ is a local martingale on $\lc 0,\tau\lc$ then $\overline X$ is a local martingale.
\end{lemma}

\begin{proof}
Supermartingale convergence implies that $G$ exists; see the proof of Proposition~A.4 in~\citet{CFR2011} and Lemma~2.1 in \citet{Larsson:Ruf:2017} for a similar statement. 
 Fatou's lemma, applied as in Theorem~1.3.15 in \citet{KS1}, yields the integrability of $\overline X_\rho$ for each $[0,\infty]$--valued stopping time $\rho$, as well as the supermartingale property of $\overline{X}$.
Now, define a sequence of stopping times $(\rho_m)_{m \in \N}$ by
\[
\rho_m = \inf\{ t\ge 0: |\overline{X}_t| > m \}
\]
and note that $\bigcup_{m\in\N}\{\rho_m=\infty\}=\Omega$.  Thus, $\overline{X}$ is stationarily locally integrable, with the corresponding sequence $(|\overline{X}_{\rho_m}| + m)_{m \in \N}$ of integrable random variables.

Assume now that $X$ is a local martingale and, without loss of generality, that  $\overline{X}^{\tau_n}$ is a uniformly integrable martingale for each $n \in \N$. 
Fix $m \in \N$ and note that  $\lim_{n\to\infty} \overline{X}_{\rho_m\wedge\tau_n}=\overline{X}_{\rho_m}$. Next, the inequality $|\overline{X}_{\rho_m \wedge \tau_n}| \leq |\overline{X}_{\rho_m}| + m$ for each $n \in \N$ justifies an application of dominated convergence as follows:
\[
\E\left[\overline{X}_{\rho_m}\right] = \E\left[\lim_{n \to \infty} \overline{X}_{\rho_m \wedge \tau_n}\right]  = \lim_{n \to \infty} \E\left[\overline{X}_{\rho_m \wedge \tau_n}\right] = 0.
\]
Hence, $\overline{X}$ is a local martingale, with localizing sequence $(\rho_m)_{m \in \N}$.
\end{proof}

For the proof of the next lemma, we are not allowed to use Corollary~\ref{C:convergence_QV}, as it relies on Theorem~\ref{T:conv}, which we have not yet proved.
\begin{lemma}[Continuous case] \label{L:CC}
Let $X$ be a continuous local martingale on $\lc0,\tau\lc$. If $[X,X]_{\tau-}<\infty$ then the limit $\lim_{t\to\tau}X_t$  exists in $\R$.
\end{lemma}

\begin{proof}
See Exercise~IV.1.48 in \citet{RY}.
\end{proof}

The next lemma will serve as a tool to handle truncated jump measures.

\begin{lemma}[Bounded jumps] \label{L:BJ}
Let $\mu$ be an integer-valued random measure such that $\mu(\R_+ \times [-1,1]^c) = 0$, and let $\nu$ be its compensator. Assume either $x^2*\mu_{\infty-}$ or $x^2*\nu_{\infty-}$ is finite. Then so is the other one, we have $x\in G_{\rm loc}(\mu)$, and the limit $\lim_{t\to\infty}x*(\mu-\nu)_t$  exists in $\R$.
\end{lemma}

\begin{proof}
First, the condition on the support of $\mu$ implies that both $x^2*\mu$ and $x^2*\nu$ have jumps bounded by one. 
Now, let $\rho_n$ be the first time $x^2*\nu$ crosses some fixed level $n \in \N$, and consider the local martingale $F=x^2*\mu-x^2*\nu$. Since $F^{\rho_n}\ge-n-1$, the supermartingale convergence theorem implies that $F^{\rho_n}_{\infty-}$ exists in $\R$, whence $x^2*\mu_{\infty-}=F_{\infty-}+x^2*\nu_{\infty-}$ exists and is finite on $\{\rho_n=\infty\}$. This yields
\begin{equation*}
\left\{ x^2*\nu_{\infty-}<\infty\right\}  \subset\left\{ x^2*\mu_{\infty-}<\infty\right\}.
\end{equation*}
The reverse inclusion is proved by interchanging $\mu$ and $\nu$ in the above argument.

Next, the local boundedness of $x^2*\nu$ implies that $x*(\mu-\nu)$ is well-defined and a local martingale with $\langle x*(\mu-\nu),x*(\mu-\nu)\rangle\le x^2*\nu$; see Theorem~II.1.33 in \citet{JacodS}. Hence, for each $n \in \N$, with $\rho_n$ as above, $x*(\mu-\nu)^{\rho_n}$ is a uniformly integrable martingale and thus convergent. Therefore $x*(\mu-\nu)$ is convergent on the set $\{\rho_n=\infty\}$, which completes the argument.
\end{proof}

\subsection{Proof of Theorem~\ref{T:conv}}

We start by proving that {\ref{T:conv:a}}  yields that $\Delta X$ is stationarily locally integrable on $D$. By localization, in conjunction with Lemma~\ref{L:ELI}\ref{L:ELI:2}, we may assume that $(\Delta X)^-\wedge X^-\le\Theta$ for some integrable random variable $\Theta$ and that $\sup_{t<\tau} |X_t| < \infty$. With $\rho_n=\inf\{t \geq 0:  X_t\le-n\}$ we have $X^{\rho_n}\ge -n-(\Delta X_{\rho_n})^-\1{\rho_n<\tau}$ and $X^{\rho_n}\ge -n-X_{\rho_n}^-\1{\rho_n<\tau}$. Hence $X^{\rho_n}  \geq -n-\Theta$ and thus, by Lemma~\ref{L:SMC}, $X^{\rho_n}$ is stationarily locally integrable and Lemma~\ref{L:ELI}\ref{L:ELI:3} yields that $\Delta X^{\rho_n}$ is as well, for each $n \in \N$. We have $\bigcup_{n\in\N}\{\rho_n=\tau\}=\Omega$, and another application of Lemma~\ref{L:ELI}\ref{L:ELI:2} yields the implication.

We now verify the claimed implications.

{\ref{T:conv:a}} $\Longrightarrow$ {\ref{T:conv:a'}}:  Obvious.

{\ref{T:conv:a'}} $\Longrightarrow$ {\ref{T:conv:a}}:   By localization we may assume that $(\Delta X)^-\wedge X^-\le\Theta$ for some integrable random variable $\Theta$ and that $\sup_{t<\tau} X_t^- < \infty$ on $\Omega$.
 With $\rho_n=\inf\{t \geq 0:  X_t\le-n\}$ we have $X^{\rho_n}  \geq -n-\Theta$, for each $n \in \N$.
 The supermartingale convergence theorem (Lemma~\ref{L:SMC}) now implies that $X$ converges.
 
{\ref{T:conv:a}} $\Longrightarrow$ {\ref{T:conv:b'}}: This is an application of Lemma~\ref{L:ELI}\ref{L:ELI:3}, after recalling that {\ref{T:conv:a}} implies that $\Delta X$ is stationarily locally integrable on $D$.

{\ref{T:conv:b'}} $\Longrightarrow$ {\ref{T:conv:a}}: This is an application of a localization argument and the supermartingale convergence theorem stated in Lemma~\ref{L:SMC}.

{\ref{T:conv:a}} $\Longrightarrow$ {\ref{T:conv:b''}} \& {\ref{T:conv:c}} \& {\ref{T:conv:f}}:
By localization, we may assume that $|\Delta X|\le \Theta$ for some integrable random variable $\Theta$ and that $X = X^{\rho}$ with $\rho=\inf\{t \geq 0:|X_t|\ge \kappa\}$ for some fixed $\kappa \geq 0$. Next, observe that $X\ge -\kappa-\Theta$. Lemma~\ref{L:SMC} yields that $G=\lim_{t\to \tau}X_t$ exists in $\R$ and that the process $\overline X=X\oo_{\lc0,\tau\lc}+G\oo_{\lc\tau,\infty\lc}$, extended to $[0,\infty]$ by $\overline X_\infty=G$, is a supermartingale on $[0,\infty]$. Let $\overline X=\overline M-\overline A$ denote its canonical decomposition. Then $A_{\tau-}=\overline A_\infty<\infty$ and $[X,X]_{\tau-}= [\overline X,\overline X]_\infty < \infty$. Moreover, since $\overline X$ is a special semimartingale on $[0,\infty]$, Proposition~II.2.29 in~\citet{JacodS} yields $(x^2\wedge|x|)*\nu^X_{\tau-}=(x^2\wedge|x|)*\nu^{\overline X}_\infty<\infty$. Thus~{\ref{T:conv:c}} and~{\ref{T:conv:f}} hold. Now, {\ref{T:conv:b''}} follows again by an application of Lemma~\ref{L:ELI}\ref{L:ELI:3}.

{\ref{T:conv:b''}} $\Longrightarrow$ {\ref{T:conv:a}}:  By Lemma~\ref{L:ELI}\ref{L:ELI:4} we may assume that $A = 0$, so that $-X$ is a local supermartingale. The result then follows again from Lemma~\ref{L:SMC} and  Lemma~\ref{L:ELI}\ref{L:ELI:3}.

{\ref{T:conv:b'} \& \ref{T:conv:b''}} $\Longrightarrow$ {\ref{T:conv:e}}: Obvious.

{\ref{T:conv:e}} $\Longrightarrow$ {\ref{T:conv:b'}}: Obvious.

{\ref{T:conv:c}} $\Longrightarrow$ {\ref{T:conv:a}}:
The process $B=[X^c,X^c]+(x^2\wedge|x|)*\nu^X+A$ is predictable and converges on~$D$. Hence, by Lemma~\ref{L:ELI}\ref{L:ELI:4}, $B$ is  stationarily locally bounded on $D$. By localization we may thus assume that $B\le\kappa$ for some constant $\kappa>0$. Lemma~\ref{L:CC} implies that $X^c$ converges and Lemma~\ref{L:BJ} implies that $x\oo_{|x|\le 1}*(\mu^X-\nu^X)$ converges. Furthermore,
\[
\E\left[\, |x|\oo_{|x|>1}*\mu^X_{\tau-}\, \right] = \E\left[\, |x|\oo_{|x|>1}*\nu^X_{\tau-}\, \right] \le \kappa,
\]
whence $|x|\oo_{|x|>1}*(\mu^X-\nu^X)=|x|\oo_{|x|>1}*\mu^X-|x|\oo_{|x|>1}*\nu^X$ converges. We deduce that~$X$ converges. It now suffices to show that $\Delta X$ is stationarily locally integrable. Since
\begin{align*}
\sup_{t<\tau} |\Delta X_t|
\leq 1 + |x|\oo_{|x|\ge 1}*\mu^X_{\tau-},
\end{align*}
we have $\E[\sup_{t<\tau} |\Delta X_t|] \le 1+\kappa$. 

{\ref{T:conv:f}} $\Longrightarrow$ {\ref{T:conv:a}}:  By a localization argument we may assume that $(\Delta X)^- \wedge X^-\leq \Theta$ for some integrable random variable~$\Theta$.  Moreover, since $[X,X]_{\tau-}<\infty$ on $D$, $X$ can only have finitely many large jumps on $D$. Thus after further localization we may assume that $X = X^{\rho}$, where $\rho=\inf\{t \geq 0:|\Delta X_t|\ge \kappa_1\}$ for some large $\kappa_1>0$. Now, Lemmas~\ref{L:CC} and~\ref{L:BJ} imply that  $X' = X^c + x \oo_{|x| \leq \kappa_1} * (\mu^X-\nu^X)$  converges on $D$.  Hence Lemma~\ref{L:ELI}\ref{L:ELI:3} and a further localization argument let us assume that $|X'|\le\kappa_2$ for some constant $\kappa_2>0$.  Define $\widehat X = x \oo_{x<-\kappa_1} * (\mu^X-\nu^X)$ and suppose for the moment we know that~$\widehat X$ converges on~$D$. Consider the decomposition
\begin{equation} \label{eq:T:conv:ftob}
X = X' + \widehat X + x\oo_{x>\kappa_1}*\mu^X-x\oo_{x>\kappa_1}*\nu^X - A.
\end{equation}
The first two terms on the right-hand side converge on $D$, as does the third term since $X=X^\rho$. However, since $\limsup_{t\to\tau}X_t>-\infty$ on $D$ by hypothesis, this forces also the last two terms to converge on~$D$, and we deduce~\ref{T:conv:a} as desired.  It remains to prove that $\widehat X$ converges on~$D$, and for this we will rely repeatedly on the equality $X=X^\rho$ without explicit mentioning. In view of~\eqref{eq:T:conv:ftob} and the bound $|X'|\le\kappa_2$, we have
\[
\widehat X  \geq X - \kappa_2  -  x \oo_{x>\kappa_1} * \mu^X = X - \kappa_2 -(\Delta X_\rho)^+\oo_{\lc\rho,\tau\lc}.
\]
Moreover,  by definition of $\widehat X$  and $\rho$ we have $\widehat X\ge 0$ on $\lc0,\rho\lc$; hence $\widehat X\ge \Delta X_\rho\oo_{\lc\rho,\tau\lc}$. We deduce from the defintiion of $\widehat X$ that  $\widehat X^-\le X^-+\kappa_2$ and $\widehat X^-\le(\Delta X)^-$. Hence we have $\widehat X^- \leq   (\Delta X)^-  \wedge  X^-  +  \kappa_2 \leq \Theta + \kappa_2.$  Lemma~\ref{L:SMC} now implies that $\widehat X$ converges, which proves the stated implication.

{\ref{T:conv:a}} \& {\ref{T:conv:f}} $\Longrightarrow$    {\ref{T:conv:g}}: We now additionally assume that $X$ is constant on $\lc\tau_J,\tau\lc$. First, note that $\Ecal(X)$ changes sign finitely many times on $D$ since $ \oo_{x < -1} * \mu_{\tau-}^X \leq  x^2  *\mu_{\tau-}^X < \infty$ on $D$. Therefore, it is sufficient to check that $\lim_{t\to\tau}|\Ecal(X)_t|$ exists in $(0,\infty)$ on $D \cap \{\tau_J = \infty\}.$  However, this follows from the fact that  $\log|\Ecal(X)|=X - [X^c,X^c] /2 - (x - \log |1+x|)  * \mu^X$ on $\lc 0, \tau_J\lc$ and the inequality $x-\log(1+x)\le x^2$ for all $x\ge-1/2$.

{\ref{T:conv:g}}  $\Longrightarrow$    {\ref{T:conv:a'}}: Note that we have  $\lim_{t \to \tau} X_t - [X^c,X^c]_t/2 - (x - \log (1+x))\oo_{x>-1} * \mu_t^X$ exists in $\R$ on $D$, which then yields the implication.
\qed

\subsection{Proof of Theorem~\ref{T:convYX}}

The proof relies on a number of intermediate lemmas. We start with a special case of Markov's inequality that is useful for estimating conditional probabilities in terms of unconditional probabilities. This inequality is then applied in a general setting to control conditional probabilities of excursions of convergent processes.

\begin{lemma}[A Markov type inequality] \label{L:cprob}
Let $\Gcal\subset\Fcal$ be a sub-sigma-field, and let $G\in\Gcal$, $F\in\Fcal$, and $\delta>0$. Then
\[
\P\left(\oo_G\,\P( F\mid\Gcal ) \geq \delta \right)\ \le\ \frac{1}{\delta}\P(G \cap F).
\]
\end{lemma}

\begin{proof} 
We have $\P(G\cap F) = \E\left[ \oo_G\,\P(F\mid\Gcal) \right] \ge \delta\, \P\left(\oo_G\,\P(F\mid\Gcal)\geq \delta\right)$.
\end{proof}

\begin{lemma} \label{L:cprob2}
Let $\tau$ be a foretellable time, let $W$ be a measurable process on $\lc0,\tau\lc$, and let $(\rho_n)_{n\in\N}$ be a nondecreasing sequence of stopping times with $\lim_{n\to\infty}\rho_n\ge\tau$. Suppose the event
\begin{align} \label{180726}
C=\Big\{\lim_{t\to\tau} W_t=0 \textnormal{ and } \rho_n<\tau \textnormal{ for all }n\in\N\Big\}
\end{align}
lies in $\Fcal_{\tau-}$. Then for each $\varepsilon>0$,
\[
\P\left( W_{\rho_n} \le \varepsilon \mid \Fcal_{\rho_n-}\right) \ge \frac{1}{2} \quad\textnormal{for infinitely many }n\in\N
\]
holds almost surely on $C$.
\end{lemma}

\begin{proof}
By Theorem~IV.71 in~\citet{Dellacherie/Meyer:1978}, $\tau$ is almost surely equal to some predictable time $\tau'$. We may thus assume without loss of generality that $\tau$ is already predictable. Define events $F_n = \{W_{\rho_n} > \varepsilon \textnormal{ and } \rho_n<\tau\}$ and $G_n = \{\P(C\mid\Fcal_{\rho_n-}) > 1/2\}$ for each $n \in \N$ and some fixed $\varepsilon > 0$.  By Lemma~\ref{L:cprob}, we have
\begin{equation}\label{eq:L:cprob2}
\P\left( \oo_{G_n} \P( F_n \mid \Fcal_{\rho_n-}) > \frac{1}{2} \right) \le 2\,\P(G_n\cap F_n) \le 2\,\P(F_n\cap C) + 2\,\P(G_n\cap C^c).
\end{equation}
Clearly, we have $\lim_{n\to \infty}\P(F_n\cap C)= 0$. Also, since $\rho_\infty=\lim_{n\to\infty}\rho_n\ge\tau$, we have 
\[ \lim_{n\to\infty}\P(C\mid\Fcal_{\rho_n-})=\P(C\mid\Fcal_{\rho_\infty-})=\oo_C. 
\]
Thus $\oo_{G_n}=\oo_C$ for all sufficiently large $n \in \N$, and hence $\lim_{n\to\infty}\P(G_n\cap C^c)=0$ by bounded convergence. The left-hand side of~\eqref{eq:L:cprob2} thus tends to zero as $n$ tends to infinity, so that, passing to a subsequence if necessary, the Borel-Cantelli lemma yields $\oo_{G_n} \P( F_n \mid \Fcal_{\rho_n-})\le1/2$ for all but finitely many $n \in \N$. Thus, since $\oo_{G_n}=\oo_C$ eventually, we have $\P( F_n \mid \Fcal_{\rho_n-})\le1/2$ for infinitely many $n \in \N$ on $C$. Since $\tau$ is predictable we have $\{\rho_n<\tau\}\in\Fcal_{\rho_n-}$ by Theorem~IV.73(b) in \citet{Dellacherie/Meyer:1978}. Thus $\P( F_n \mid \Fcal_{\rho_n-})=\P( W_{\rho_n}>\varepsilon \mid \Fcal_{\rho_n-})$ on $C$, which yields the desired conclusion.
\end{proof}

Returning to the setting of Theorem~\ref{T:convYX}, we now show that $\gamma$ vanishes asymptotically on the event~\eqref{T:convYX:4}.

\begin{lemma} \label{L:convYX:1}
Under Assumption~\ref{A:2}, we have $\lim_{t\to\tau}\gamma_t=0$ on~\eqref{T:convYX:4}.
\end{lemma}

\begin{proof}
As in the proof of Lemma~\ref{L:cprob} we may assume that $\tau$ is predictable. We now argue by contradiction. To this end, assume there exists $\varepsilon>0$ such that $\P(D)>0$ where
\[
 D=\{\gamma_t\ge 2 \varepsilon\textnormal{ for infinitely many $t$}\} \cap \eqref{T:convYX:4}.
\]
 Let $(\rho_n)_{n\in\N}$ be a sequence of predictable times covering the predictable set $\{\gamma\ge2\varepsilon\}$. By \eqref{eq:DY} and since $X$ and $Y$ are c\`adl\`ag, any compact subset of $[0,\tau)$ can only contain finitely many time points~$t$ for which $\gamma_t\ge 2\varepsilon$. We may thus take the $\rho_n$ to satisfy $\rho_n<\rho_{n+1}<\tau$ on $D$ for all $n\in\N$, as well as $\lim_{n\to\infty}\rho_n\ge\tau$.

We now have, for each $n\in\N$ on $\{\rho_n<\tau\}$,
\begin{align*}
0
&= \int x\nu^X(\{\rho_n\}, \dd x) 
\le -(1-e^{-\varepsilon}) \P\left(\Delta X_{\rho_n} \le e^{-\varepsilon}-1\mid\Fcal_{\rho_n-}\right) + \int x\oo_{x>0}\, \nu^X(\{\rho_n\},\dd x) \\
&\le -(1-e^{-\varepsilon}) \P\left(\Delta Y_{\rho_n} \le \varepsilon\mid\Fcal_{\rho_n-}\right) + \int x\oo_{x>0}\, \nu^X(\{\rho_n\},\dd x),
\end{align*}
where the equality uses the local martingale property of $X$, the first inequality is an elementary bound involving Equation~II.1.26 in~\citet{JacodS}, and the second inequality follows from~\eqref{eq:DY}. 

Thus on $D$,
\begin{equation} \label{eq:convYX:1:1}
x \oo_{x \geq 0 \vee (e^{\varepsilon-\gamma}-1)}*\nu^X_{\tau-}
\ge \sum_{n\in\N}\int x\oo_{x>0}\, \nu^X(\{\rho_n\},\dd x) 
\ge (1-e^{-\varepsilon})\sum_{n\in\N}\P\left(\Delta Y_{\rho_n} \le \varepsilon\mid\Fcal_{\rho_n-}\right).
\end{equation}
With $W = \Delta Y$, Lemma~\ref{L:cprob2} implies that the right-hand side of~\eqref{eq:convYX:1:1} is infinite almost surely on $ C \supset D$, where $C$ is given in \eqref{180726}. 

We now argue that the left-hand side of~\eqref{eq:convYX:1:1} is finite almost surely on \eqref{T:convYX:4} $\supset D$, yielding the contradiction. To this end, since $\lim_{t\to\tau}\Delta Y_t=0$ on~\eqref{T:convYX:4}, we have $\oo_{x > e^{\varepsilon - \gamma} - 1} * \mu^X_{\tau-} < \infty$ on~\eqref{T:convYX:4}. Lemma~\ref{L:BJ}  applied to the random measure
\[
\mu=\oo_{0 \vee (e^{\varepsilon-\gamma}-1)\le x\le \kappa} \oo_{\lc 0, \tau\lc}\,\mu^X
\]
yield $x \oo_{0 \vee (e^{\varepsilon-\gamma}-1)\le x\le \kappa}*\nu^X_{\tau-}<\infty$; here $\kappa$ is as in Theorem~\ref{T:convYX}.  Since also $x\oo_{x>\kappa}*\nu^X_{\tau-}<\infty$ on~\eqref{T:convYX:4} by definition, the left-hand side of~\eqref{eq:convYX:1:1} is finite. 
\end{proof}

\begin{lemma} \label{L:convYX:0 new}
Fix $\varepsilon \in (0,1)$. Under Assumption~\ref{A:2}, we have
\[
[X^c, X^c]_{\tau-} + (\log(1+x)+\gamma)^2 \oo_{|x| \leq \varepsilon} * \nu^X_{\tau-}  - \log(1+x) \oo_{x \leq -\varepsilon} * \nu^X_{\tau-} + x  \oo_{x \geq  \varepsilon} * \nu^X_{\tau-}  < \infty
\]
 on~\eqref{T:convYX:4}.
 \end{lemma}

\begin{proof}
By Lemma~\ref{L:convYX}  condition \ref{T:conv:a} in Theorem~\ref{T:conv} holds with $X$ replaced by $-Y$. Using the equivalence with Theorem~\ref{T:conv}\ref{T:conv:c}, we obtain that
 $[X^c, X^c]_{\tau-}  = [Y^c, Y^c]_{\tau-} < \infty$ and 
\begin{align}\label{eq: L:convYX:0 new proof}
\Big( (\log(1+x)+\gamma)^2 \wedge |\log(1+x)+\gamma|\Big) * \nu^X_{\tau-}  + \sum_{s<\tau} (\gamma_s^2\wedge\gamma_s)\1{\Delta X_s=0}
= (y^2\wedge|y|)*\nu^Y_{\tau-}  < \infty
\end{align}
on \eqref{T:convYX:4}, where the equality in~\eqref{eq: L:convYX:0 new proof} follows from \eqref{eq:DY}.  Now, by localization, Lemma~\ref{L:convYX:1}, and Lemma~\ref{L:ELI}\ref{L:ELI:4}, we may assume that $\sup_{t < \tau} \gamma_t$ is bounded.  We then obtain from~\eqref{eq: L:convYX:0 new proof} that $(\log(1+x)+\gamma)^2 \oo_{|x| \leq \varepsilon} * \nu^X_{\tau-} < \infty$ on \eqref{T:convYX:4}. 

Next, note that
\begin{align*}  
  -\log(1+x)\oo_{x\le -\varepsilon} * \nu^X_{\tau-} &=   -\log(1+x)\oo_{x\le-\varepsilon}  \oo_{\{\gamma<-\log(1-\varepsilon)/2\}}  * \nu^X_{\tau-} \\
  	&\qquad +  \sum_{t < \tau} \int -\log(1+x)\oo_{x\le-\varepsilon}  \oo_{\{\gamma_t\geq -\log(1-\varepsilon)/2\}}\,  \nu^X(\{t\},\dd x) < \infty
\end{align*}
on \eqref{T:convYX:4}. Indeed, an argument based on \eqref{eq: L:convYX:0 new proof}  shows that the first summand is finite. The second summand is also finite since it consists of finitely many terms due to Lemma~\ref{L:convYX:1}, each of which is finite. The latter follows since $(x-\log(1+x))*\nu^X$ is a finite-valued process by assumption and $\int|x|\nu^X(\{t\},dx)<\infty$ for all $t<\tau$ due to the local martingale property of $X$.
Finally, a calculation based on \eqref{eq: L:convYX:0 new proof} yields $x\oo_{\varepsilon\le x\le \kappa}*\nu^X_{\tau-}<\infty$ on the event \eqref{T:convYX:4}, where $\kappa$ is as in Theorem~\ref{T:convYX}. This, together with the definition of \eqref{T:convYX:4}, implies that $x  \oo_{x \geq  \varepsilon} * \nu^X_{\tau-}  < \infty$ there, completing the proof.
\end{proof}

We are now ready to verify the claimed inclusions of Theorem~\ref{T:convYX}.

\eqref{T:convYX:1} $\subset$ \eqref{T:convYX:2}: The implication \ref{T:conv:a} $\Longrightarrow$ \ref{T:conv:g} of Theorem~\ref{T:conv} shows that $\Ecal(X)_{\tau-}>0$ on~\eqref{T:convYX:1}. The desired inclusion now follows from~\eqref{eq:VV}.

\eqref{T:convYX:2} $\subset$ \eqref{T:convYX:1}: By the inclusion \eqref{T:conv2:2} $\subset$ \eqref{T:conv2:1} of Corollary~\ref{C:conv2} and the implication \ref{T:conv:a} $\Longrightarrow$ \ref{T:conv:g} of Theorem~\ref{T:conv}, $X$ converges and $\Ecal(X)_{\tau-}>0$ on \eqref{T:convYX:2}. Hence by \eqref{eq:VV}, $Y$ also converges on \eqref{T:convYX:2}.

\eqref{T:convYX:1}  $\cap$ \eqref{T:convYX:2} $\subset$ \eqref{T:convYX:3}: Obvious.

\eqref{T:convYX:3} $\subset$ \eqref{T:convYX:2}: The inclusion \eqref{T:conv2:1} $\subset$ \eqref{T:conv2:2} of Corollary~\ref{C:conv2} implies $[X^c,X^c]_{\tau-}+(x^2\wedge|x|)*\nu^X_{\tau-}<\infty$ on~\eqref{T:convYX:3}. Since also $-\log(1+x)\oo_{x\le -\eta}*\nu^X_{\tau-} < \infty$ on~\eqref{T:convYX:3} by definition, the desired inclusion follows.

\eqref{T:convYX:1} $\cap$ \eqref{T:convYX:2} $\subset$ \eqref{T:convYX:4}: Obvious.

\eqref{T:convYX:4} $\subset$ \eqref{T:convYX:1}: We need to show that $X$ converges on \eqref{T:convYX:4}. By Theorem~\ref{T:conv} it is sufficient to argue that $[X^c,X^c]_{\tau-}+(x^2\wedge|x|)*\nu^X_{\tau-}<\infty$ on \eqref{T:convYX:4}.  Lemma~\ref{L:convYX:0 new} yields directly that $[X^c,X^c]_{\tau-}< \infty$, so we focus on the jump component. To this end, using that $\int x\nu^X(\{t\},dx)=0$ for all $t<\tau$, we first observe that, for fixed $\varepsilon\in (0,1)$,
\begin{align*}
\gamma_t
&= \int (x-\log(1+x)) \,\nu^X(\{t\},\dd x) \\
&\le \frac{1}{1-\varepsilon} \int x^2\oo_{|x|\le \varepsilon}\,\nu^X(\{t\},\dd x) + \int x\oo_{x>\varepsilon}\,\nu^X(\{t\},\dd x)  +  \int -\log(1+x)\oo_{x<-\varepsilon}\,\nu^X(\{t\},\dd x)
\end{align*}
for all $t<\tau$.
Letting $\Theta_t$ denote the last two terms for each $t < \tau$, Lemma~\ref{L:convYX:0 new} implies that $\sum_{t<\tau}\Theta_t<\infty$, and hence also $\sum_{t<\tau}\Theta_t^2<\infty$, hold on~\eqref{T:convYX:4}. Furthermore, the inequality $(a+b)^2\le 2a^2 + 2b^2$ yields that
\begin{align} \label{eq: sum gamma2}
\sum_{t<\tau_n}  \gamma_t^2
&\le  \frac{2}{(1-\varepsilon)^2} \sum_{t<\tau_n}  \left( \int x^2\oo_{ |x| \le \varepsilon}\,\nu^X(\{t\},dx) \right)^2 + 2 \sum_{t<\tau_n}  \Theta_t^2 
 \le  \frac{2 \varepsilon^2}{(1-\varepsilon)^2}  x^2\oo_{|x|\le \varepsilon} * \nu^X_{\tau_n} + 2 \sum_{t<\tau}  \Theta_t^2
\end{align} 
for all $n \in \N$, where $(\tau_n)_{n\in\N}$ denotes an announcing sequence for~$\tau$.

Also observe that, for all $n \in \N$,
\begin{align*}
\frac{1}{16}  x^2 \oo_{|x|\le \varepsilon} * \nu^X_{\tau_n} 
&\leq (\log(1+x))^2 \oo_{|x|\leq \varepsilon} * \nu^X_{\tau_n} 
\leq 2 (\log(1+x) + \gamma)^2 \oo_{|x| \leq \varepsilon} * \nu^X_{\tau_n} + 2 \sum_{t \leq \tau_n} \gamma_t^2,
\end{align*}
which yields, thanks to \eqref{eq: sum gamma2},
\begin{align*}
	\left(\frac{1}{16} - \frac{4 \varepsilon^2}{(1-\varepsilon)^2} \right) x^2 \oo_{|x|\le \varepsilon} * \nu^X_{\tau_n} \leq 2 (\log(1+x) + \gamma)^2 \oo_{|x| \le \varepsilon} * \nu^X_{\tau_n} + 4 \sum_{t<\tau}  \Theta_t^2.
\end{align*}
Choosing $\varepsilon$ small enough and letting $n$ tend to infinity, we obtain that $x^2 \oo_{|x|\le \varepsilon} * \nu^X_{\tau-} < \infty$ on \eqref{T:convYX:4} thanks to Lemma~\ref{L:convYX:0 new}. The same lemma also yields  $|x| \oo_{|x|\geq \varepsilon} * \nu^X_{\tau-} < \infty$, which concludes the proof.
\qed

\section{Counterexamples}  \label{S:examp}

In this section we collect several examples of local martingales that illustrate the wide range of asymptotic behavior that can occur. This showcases the sharpness of the results in Section~\ref{S:convergence}.

\subsection{Random walk with large jumps}  \label{A:SS:lack}

Choose a sequence $(p_n)_{n \in \N}$ of real numbers such that $p_n\in(0,1)$ and $\sum_{n =1}^\infty p_n < \infty$. Moreover, choose a sequence $(x_n)_{n \in \N}$ of real numbers.
Then let $(\Theta_n)_{n \in \N}$ be a sequence of independent random variables with $\P(\Theta_n = 1) = p_n$ and $\P(\Theta_n = 0) = 1-p_n$ for all $n \in N$.  Now define a process $X$ by
\begin{align*}
	X_t = \sum_{n =1}^{[t]} x_n \left(1 - \frac{\Theta_n}{p_n}\right),
\end{align*}
where $[t]$ is the largest integer less  than or equal to  $t$, and let $\F$ be its natural filtration. Clearly $X$ is a locally bounded martingale. The Borel-Cantelli lemma implies that  $\Theta_n$ is nonzero for only finitely many $n \in \N$, almost surely, whence for all sufficiently large $n \in \N$ we have $\Delta X_n = x_n$. By choosing a suitable sequence $(x_n)_{n \in \N}$ one may therefore achieve essentially arbitrary asymptotic behavior. This construction was inspired by an example due to George Lowther that appeared on his blog Almost Sure on December 20, 2009.

\begin{lemma} \label{L:ex1}
With the notation of this subsection, $X$ satisfies the following properties:
\begin{enumerate}
\item\label{L:ex1:1} $\lim_{t \to \infty}  X_t$ exists in $\R$ if and only if $\lim_{m\to \infty} \sum_{n = 1}^m x_n$ exists in $\R$.
\item\label{L:ex1:3} $(1 \wedge x^2) * \mu^X_{\infty-} < \infty$ if and only if $[X,X]_{\infty-} < \infty$ if and only if $\sum_{n = 1}^\infty x_n^2 < \infty$.
\item\label{L:ex1:2} $X$ is a semimartingale on $[0,\infty]$ if and only if $(x^2\wedge|x|)*\nu^X_{\infty-}<\infty$ if and only if $X$ is a uniformly integrable martingale if and only if  $\sum_{n = 1}^\infty |x_n| < \infty$.
\end{enumerate}
\end{lemma}
\begin{proof}
The statements in \ref{L:ex1:1} and \ref{L:ex1:3} follow from the Borel-Cantelli lemma. For \ref{L:ex1:2}, note that $|X_t| \le \sum_{n\in\N} |x_n|  (1 + \Theta_n/p_n)$ for all $t\ge0$. Since
\[
\E\left[\sum_{n = 1}^\infty |x_n|  \left(1 + \frac{\Theta_n}{p_n}\right)\right]=2\sum_{n = 1}^\infty |x_n|,
\]
the condition $\sum_{n = 1}^\infty |x_n| < \infty$ implies that $X$ is a uniformly integrable martingale, which implies that~$X$ is a special semimartingale on $[0,\infty]$, or equivalently that $(x^2\wedge|x|)*\nu^X_{\infty-}<\infty$ (see Proposition~II.2.29 in~\citet{JacodS}), which implies that $X$ is a semimartingale on $[0,\infty]$. It remains to show that this implies $\sum_{n = 1}^\infty |x_n| < \infty$. We prove the contrapositive, and assume $\sum_{n = 1}^\infty |x_n| = \infty$. Consider the bounded predictable process $H = \sum_{n=1}^\infty (\oo_{x_n>0} - \oo_{x_n<0}) \oo_{\lc n\rc}$. If $X$ were a semimartingale on $[0,\infty]$, then $(H\cdot X)_{\infty-}$ would be well-defined and finite. However, by Borel-Cantelli, $H\cdot X$ has the same asymptotic behavior as $\sum_{n=1}^\infty |x_n|$ and thus diverges. Hence $X$ is not a semimartingale on~$[0,\infty]$.
\end{proof}

Martingales of the above type can be used to illustrate that much of Theorem~\ref{T:conv} and its corollaries fails if one drops stationarily local integrability of $(\Delta X)^-\wedge X^-$. We now list several such counterexamples.

\begin{example} \label{E:P1}
We use the notation of this subsection.
\begin{enumerate}
\item \label{E:P1:1} Let $x_n = (-1)^n/\sqrt{n}$ for all $n \in \N$. Then
\[
\P\Big(\lim_{t \to \infty} X_t \textnormal{ exists in $\R$}\Big)  = \P\left([X,X]_{\infty-} = \infty\right) =  \P\left(x^2 \oo_{|x|<1}* \nu^X_{\infty-} = \infty\right) = 1.
\]
Thus the implications \ref{T:conv:a} $\Longrightarrow$ \ref{T:conv:c} and \ref{T:conv:a} $\Longrightarrow$ \ref{T:conv:f} in Theorem~\ref{T:conv} fail without the integrability condition on $(\Delta X)^-\wedge X^-$. Furthermore, by setting $x_1=0$ but leaving  $x_n$ for all $n \geq 2$ unchanged, and ensuring that $p_n\ne x_n/(1+x_n)$ for all $n \in \N$, we have $\Delta X\ne-1$. Thus, $\Ecal(X)_t=\prod_{n=1}^{[t]} (1+\Delta X_n)$ is nonzero for all $t$. Since $\Delta X_n=x_n$ for all sufficiently large $n \in \N$, $\Ecal(X)$ will eventually be of constant sign. Moreover, for any $n_0\in\N$ we have
\[
\lim_{m\to\infty}\sum_{n=n_0}^m \log(1+x_n)\le \lim_{m\to\infty}\sum_{n=n_0}^m \left(x_n - \frac{x_n^2}{4}\right)=-\infty.
\]
It follows that $\P( \lim_{t\to\infty}\Ecal(X)_t=0) = 1$, showing that the implication \ref{T:conv:a} $\Longrightarrow$ \ref{T:conv:g} in Theorem~\ref{T:conv} fails without the integrability condition on $(\Delta X)^-\wedge X^-$.

\item  Part~\ref{E:P1:1} illustrates that the implications \ref{T:conv:a'} $\Longrightarrow$ \ref{T:conv:c},  \ref{T:conv:a'} $\Longrightarrow$ \ref{T:conv:f}, and \ref{T:conv:a'} $\Longrightarrow$ \ref{T:conv:g} in Theorem~\ref{T:conv} fail without the integrability condition on $(\Delta X)^-\wedge X^-$.
We now let $x_n = 1$ for all $n \in \N$. Then $\P(\lim_{t \to \infty} X_t = \infty)  = 1$, which illustrates that also \ref{T:conv:a'} $\Longrightarrow$ \ref{T:conv:a} in that theorem fails without integrability condition.

\item We now fix a sequence $(x_n)_{n \in \N}$ such that $|x_n| = 1/n$  but $g: m \mapsto \sum_{n = 1}^m x_n$ oscillates with $\liminf_{m \to \infty} g(m) = -\infty$ and  $\limsup_{m \to \infty} g(m) = \infty$.  This setup illustrates that \ref{T:conv:f} $\Longrightarrow$ \ref{T:conv:a} and \ref{T:conv:f} $\Longrightarrow$ \ref{T:conv:a'}  in Theorem~\ref{T:conv}  fail without the integrability condition on $(\Delta X)^- \wedge X^-$.  Moreover, by Lemma~\ref{L:ex1}\ref{L:ex1:2} the implication \ref{T:conv:f} $\Longrightarrow$ \ref{T:conv:c} fails without the additional integrability condition. The same is true for the implication \ref{T:conv:f} $\Longrightarrow$ \ref{T:conv:g}, since $\log \mathcal{E}(X) \leq X$.

\item Let $x_n=e^{(-1)^n/\sqrt{n}}-1$ and suppose $p_n\ne x_n/(1+x_n)$ for all $n \in \N$ to ensure $\Delta X\ne-1$. Then
\[
\P\Big( \lim_{t\to\infty}\Ecal(X)_t \textnormal{ exists in }\R\setminus\{0\}\Big) = \P\Big(\lim_{t \to \infty} X_t = \infty\Big) = \P\Big([X,X]_{\infty-}= \infty\Big) = 1.
\]
Indeed, $\lim_{m\to\infty}\sum_{n=1}^m \log(1+x_n)=\lim_{m\to\infty}\sum_{n=1}^m (-1)^n/\sqrt{n}$ exists in $\R$, implying that $\Ecal(X)$ converges to a nonzero limit. Moreover,
\[
\lim_{m\to\infty}\sum_{n=1}^m x_n \ge \lim_{m\to\infty}\sum_{n=1}^m \left(\frac{(-1)^n}{\sqrt{n}}+\frac{1}{4n}\right)=\infty,
\]
whence $X$ diverges. Since $\sum_{n=1}^\infty x_n^2 = \infty$, we obtain that $[X,X]$ also diverges. Thus the implications \ref{T:conv:g} $\Longrightarrow$ \ref{T:conv:a} and \ref{T:conv:g} $\Longrightarrow$ \ref{T:conv:f} in Theorem~\ref{T:conv} fail without the integrability condition on $(\Delta X)^-\wedge X^-$. So does the implication \ref{T:conv:g} $\Longrightarrow$ \ref{T:conv:c} due to Lemma~\ref{L:ex1}\ref{L:ex1:2}. Finally, note that the implication \ref{T:conv:g} $\Longrightarrow$ \ref{T:conv:a'} holds independently of any integrability conditions since $\log \mathcal{E}(X) \leq X$.

\item Let $x_n=-1/n$ for all $n\in\N$. Then $[X,X]_{\infty-}<\infty$ and $(\Delta X)^-$ is stationarily locally integrable, but $\lim_{t\to\infty}X_t=-\infty$. This shows that the condition involving limit superior is needed in Theorem~\ref{T:conv}\ref{T:conv:f}, even if $X$ is a martingale. We further note that if $X$ is Brownian motion, then $\limsup_{t\to\infty}X_t>-\infty$ and $(\Delta X)^-=0$, but $[X,X]_{\infty-}=\infty$. Thus some condition involving the quadratic variation is also needed in Theorem~\ref{T:conv}\ref{T:conv:f}.

\item Note that choosing $x_n= (-1)^n/n$ for each $n\in \N$ yields a locally bounded martingale~$X$ with $[X,X]_{\infty-}<\infty$, $X_\infty = \lim_{t \to \infty} X_t$ exists, but $X$ is not a semimartingale on $[0,\infty]$.   This contradicts statements in the literature which assert that a semimartingale that has a limit is a semimartingale on the extended interval.  This example also illustrates that the implications \ref{T:conv1:a} $\Longrightarrow$ \ref{C:conv001:d} and \ref{T:conv1:a} $\Longrightarrow$ \ref{C:conv001:e}
in Corollary~\ref{C:conv001} fail without additional integrability condition. For the sake of completeness, Example~\ref{ex:semimartingale} below illustrates that the integrability condition in Corollary~\ref{C:conv001}\ref{C:conv001:e} is not redundant either. \qed
\end{enumerate}
\end{example}

\begin{remark}
Many other types of behavior can be generated within the setup of this subsection. For example, by choosing the sequence $(x_n)_{n \in \N}$ appropriately we can obtain a martingale $X$ that converges nowhere, but satisfies $\P(\sup_{t\geq 0} |X_t| < \infty)=1$. We can also choose $(x_n)_{n \in \N}$ so that, additionally, either $\P([X,X]_{\infty-} = \infty)=1$ or $\P([X,X]_{\infty-} <\infty)=1$.
\qed
\end{remark}

\begin{example} \label{ex:ui}
The assumption in Corollary~\ref{C:convYX2} \ref{T:convYX2:b} cannot be weakened to $L^1$--boundedness. To see this, within the setup of this subsection, let $x_n=-1/2$ for all $n \in \N$. Then $\Delta X\ge-1/2$. Moreover,  we claim that the sequence $(p_n)_{n\in\N}$ can be chosen so that
\begin{align}  \label{ex:ui:eq1}
	\sup_{\sigma\in\Tcal} \E\left[ e^{c\log(1+x) * (\mu^X-\nu^X)_\sigma}  \right] < \infty
\end{align}
for each $c<1$,
while, clearly, $\P(\lim_{t \to \infty} X_t = -\infty)=1$. This shows that the implication \ref{T:convYX2:b} $\Longrightarrow$ \ref{T:convYX2:a} in Corollary~\ref{C:convYX2}, with $c<1$, fails without the tail condition on $\nu^X$.

To obtain~\eqref{ex:ui:eq1}, note that $Y=\log(1+x) * (\mu^X-\nu^X)$ is a martingale, so that $e^{cY}$ is a submartingale, whence $\E[e^{cY_\sigma}]$ is nondecreasing in~$\sigma$. Since the jumps of $X$ are independent, this yields
\begin{align*}
	\sup_{\sigma\in\Tcal}  \E\left[ e^{c \log(1+x) * (\mu^X-\nu^X)_\sigma}  \right]
\le  \prod_{n=1}^\infty \E\left[ (1+\Delta X_n)^c \right] e^{-c\,\E[\log(1+\Delta X_n)]} =: \prod_{n=1}^\infty e^{\kappa_n}.
\end{align*}
We have $\kappa_n\ge0$ by Jensen's inequality, and a direct calculation yields
\begin{align*}
\kappa_n =& \log\E\left[ (1+\Delta X_n)^{c} \right] - c\,\E[\log(1+\Delta X_n)] 
\le \log\left( 2 p_n^{1-c} +1\right) - c\,p_n \log(1+p_n^{-1})
\end{align*}
for all $c <1$.
Let us now fix a sequence $(p_n)_{n \in \N}$ such that the following inequalities hold for all $n \in \N$:
\begin{align*}
	p_n \log(1+p_n^{-1})  \leq \frac{1}{n^3} \qquad \textnormal{and} \qquad
	p_n \leq \frac{1}{2^n}\left( e^{n^{-2}}  -1\right)^n. 
\end{align*}
This is always possible. Such a sequence satisfies $\sum_{n\in\N}p_n<\infty$ and results in $\kappa_n \leq 2/n^2$ for all $n \geq (-c) \vee (1/(1-c))$, whence $\sum_{n\in\N}\kappa_n<\infty$. This yields the assertion.
\qed
\end{example}

\subsection{Quasi-left continuous one-jump martingales}  \label{A:SS:one}

We now present examples based on a martingale $X$ which, unlike in Subsection~\ref{A:SS:lack}, has one single jump that occurs at a totally inaccessible stopping time. In particular, the findings of Subsection~\ref{A:SS:lack} do not rely on the fact that the jump times there are predictable.

Let $\lambda, \gamma:\mathbb R_+\to\mathbb R_+$ be two continuous nonnegative functions. Let $\Theta$ be a standard exponential random variable and define $\rho = \inf\{t\ge 0: \int_0^t\lambda(s) ds \ge \Theta\}$. Let $\F$ be the filtration generated by the indicator process $\oo_{\lc\rho,\infty\lc}$, and define a process $X$ by
\[
X_t = \gamma(\rho)\1{\rho\le t} - \int_0^t \gamma(s)\lambda(s)\1{s<\rho}\dd s.
\]
Note that $X$ is the integral of $\gamma$ with respect to $\1{\rho\le t}-\int_0^{t\wedge\rho}\lambda_s\dd s$ and is  a martingale. Furthermore, $\rho$ is totally inaccessible. This construction is sometimes called the {\em Cox construction}. Furthermore, the jump measure $\mu^X$ and corresponding compensator $\nu^X$ satisfy
\[
F*\mu^X = F(\rho,\gamma(\rho))\oo_{\lc \rho, \infty\lc}, \qquad
F*\nu^X_t = \int_0^{t\wedge\rho}F(s,\gamma(s))\lambda(s) \dd s
\]
for all $t \geq 0$,
where $F$ is any nonnegative predictable function. We will study such martingales when $\lambda$ and $\gamma$ posses certain integrability properties, such as the following:
\begin{align}
&\int_0^\infty\lambda(s)\dd s <\infty;   \label{A:eq:int1}\\
&\int_0^\infty\gamma(s)\lambda(s)\dd s =\infty;  \label{A:eq:int2}\\
&\int_0^\infty (1+\gamma(s))^c\lambda(s) \dd s <\infty \quad \textnormal{ for all }c<1.  \label{A:eq:int3}
\end{align}
For instance, $\lambda(s) = 1/(1+s)^2$ and $\gamma(s)=s$ satisfy all three properties.

\begin{example} \label{E:2}
The limit superior condition in Theorem~\ref{T:conv} is essential, even if $X$ is a local martingale. Indeed, with the notation of this subsection, let $\lambda$ and $\gamma$ satisfy \eqref{A:eq:int1} and \eqref{A:eq:int2}. Then
\begin{align*}
&\P\left([X,X]_{\infty-} +  (x^2\wedge 1) * \nu^X_{\infty-}  < \infty\right) = \P\Big(\sup_{t\ge 0} X_t < \infty\Big)=1; \\
&\P \Big(\limsup_{t\to \infty} X_t = - \infty\Big) > 0.
\end{align*}
This shows that finite quadratic variation does not prevent a martingale from diverging; in fact, $X$ satisfies $\{[X,X]_{\infty-} =0\} = \{\limsup_{t\to \infty}X_t = - \infty\}$. The example also shows that one cannot replace $(x^2\wedge |x|)*\nu^X_{\infty-}$ by $(x^2\wedge 1) * \nu^X_{\infty-}$ in \eqref{T:conv2:2}. Finally, it illustrates in the quasi-left continuous case that diverging local martingales need not oscillate, in contrast to continuous local martingales.

To prove the above claims, first observe that $[X,X]_{\infty-} = \gamma(\rho)^2 \1{\rho<\infty} < \infty$ and $\sup_{t\ge 0} X_t\le\gamma(\rho)\1{\rho<\infty}<\infty$ almost surely. Next, we get $\P(\rho=\infty) = \exp({-\int_0^\infty\lambda(s) \dd s})>0$ in view of~\eqref{A:eq:int1}. We conclude by observing that $\lim_{t \to \infty} X_t = - \lim_{t \to \infty}  \int_0^t \gamma(s)\lambda(s) \dd s = -\infty$ on the event $\{\rho=\infty\}$ due to~\eqref{A:eq:int2}.
\qed
\end{example}

\begin{example}
Example~\ref{E:2} can be refined to yield a martingale with a single positive jump, that diverges without oscillating, but has infinite quadratic variation. To this end, extend the probability space to include a Brownian motion $B$ that is independent of~$\Theta$, and suppose $\F$ is generated by $(\oo_{\lc\rho,\infty\lc},B)$. The construction of $X$ is unaffected by this. In addition to \eqref{A:eq:int1} and \eqref{A:eq:int2}, let $\lambda$ and $\gamma$ satisfy
\begin{equation} \label{eq:E:3:1}
\lim_{t\to\infty} \frac{\int_0^t \gamma(s)\lambda(s) \dd s }{\sqrt{2 t \log \log t}} = \infty.
\end{equation}
For instance, take $\lambda(s)=1/(1+s)^2$ and $\gamma(s)=1/\lambda(s)$. Then the martingale $X' = B + X$ satisfies
\begin{equation} \label{eq:E:3:2}
\P\left([X',X']_{\infty-} = \infty\right)= 1 \qquad\textnormal{and}\qquad \P\Big(\sup_{t\ge 0} X'_t < \infty\Big) > 0,
\end{equation}
so that, in particular, the inclusion $\{ [X',X']_{\infty-}=\infty\} \subset \{\sup_{t\ge 0} X'_t=\infty\}$ does not hold in general. 

To prove~\eqref{eq:E:3:2}, first note that $[X',X']_{\infty-} \geq [B,B]_{\infty-} = \infty$. Next, \eqref{eq:E:3:1} and the law of the iterated logarithm yield, on the event $\{\rho=\infty\}$,
\[
\limsup_{t \to \infty} X'_t = \limsup_{t \to \infty} \Big(B_t - \int_0^t \gamma(s)\lambda(s) \dd s\Big)  \leq \limsup_{t \to \infty} \Big(2 \sqrt{2 t \log \log t}  - \int_0^t \gamma(s)\lambda(s) \dd s\Big) = -\infty.
\]
Since $\P(\rho=\infty)>0$, this implies $\P(\sup_{t\geq 0} X'_t < \infty)>0$.
\qed
\end{example}

\begin{example} \label{ex:semimartingale}
The semimartingale property does not imply that $X^- \wedge (\Delta X)^-$ is stationarily locally integrable. With the notation of this subsection, let $\lambda$ and $\gamma$ satisfy \eqref{A:eq:int1} and \eqref{A:eq:int2}, and consider the process $\widehat{X} = -\gamma(\rho) \oo_{\lc \rho, \infty\lc}$, which is clearly a semimartingale on $[0,\infty]$. On $[0,\infty)$, it has the special decomposition $\widehat{X} = \widehat{M} - \widehat{A}$, where $\widehat{M} = -X$ and $\widehat{A} = \int_0 \gamma(s) \lambda(s) \oo_{\{s < \rho\}} \dd s$. We have $\P(\widehat A_\infty = \infty) > 0$, and thus, by Corollary~\ref{C:conv2} we see that the integrability condition in Corollary~\ref{C:conv001}\ref{C:conv001:e} is non-redundant. This example also illustrates that~\eqref{eq:XMA} does not hold in general.
\qed
\end{example}

\begin{example}\label{ex:6.8}
Also in the case where $X$ is quasi-left continuous, the uniform integrability assumption in Corollary~\ref{C:convYX2} cannot be weakened to $L^1$--boundedness. We again put ourselves in the setup of this subsection and suppose $\lambda$ and $\gamma$ satisfy \eqref{A:eq:int1}--\eqref{A:eq:int3}. Then, while $X$ diverges with positive probability, it nonetheless satisfies
\begin{align}  \label{A:eq:prop6.1}
	\sup_{\sigma\in\Tcal} \E\left[ e^{c\log(1+x) * (\mu^X-\nu^X)_\sigma }\right] < \infty
\end{align}
for all $c<1$.
Indeed, if $c \leq 0$, then $$e^{c\log(1+x) * (\mu^X-\nu^X)_\sigma} \leq e^{|c| \log(1+x) * \nu^X_\rho} \leq e^{|c| \int_0^\infty \log(1+s)  \lambda(s) \dd s} < \infty $$ for all $\sigma \in \Tcal$. 
If $c \in (0,1)$, the left-hand side of~\eqref{A:eq:prop6.1} is bounded above by
\begin{align*}
\sup_{\sigma \in \mathcal T} \E\left[e^{c\log(1+x)*\mu^X_\sigma}\right]
&\le 1 + \sup_{\sigma \in \mathcal T} \E\left[(1+\gamma(\rho))^c\,\1{\rho\le\sigma}\right] 
\le 1 + \E\left[(1+x)^c*\mu^X_{\infty}\right] \\
&=1 + \E\left[(1+x)^c*\nu^X_{\infty}\right] \le 1 +  \int_0^\infty(1+\gamma(s))^c\lambda(s)\dd s < \infty,
\end{align*}
due to \eqref{A:eq:int3}. 
\qed
\end{example}

\nocite{Stricker_1981}

\setlength{\bibsep}{0.0pt}
\bibliography{aa_bib}{}
\bibliographystyle{apalike}

\end{document}